\documentclass[11pt]{amsart}
\usepackage{amscd,amssymb}
\usepackage{amsmath, amssymb, verbatim,color, amscd}
\usepackage[mathscr]{eucal}
\usepackage{tikz,graphicx}

\input xy
\xyoption{all}

\newtheorem{theorem}{Theorem}[section]
\newtheorem{lemma}[theorem]{Lemma}

\newtheorem{proposition}[theorem]{Proposition}

\renewcommand{\leq}{\leqslant}
\renewcommand{\geq}{\geqslant}

\theoremstyle{definition}

\newtheorem{example}[theorem]{Example}

\theoremstyle{definition}

\numberwithin{equation}{section}

\numberwithin{equation}{section}
 \numberwithin{figure}{section}

\author{Yuguang Zhang}
\address{Institut f\"{u}r Differentialgeometrie,  Gottfried Wilhelm Leibniz Universit\"{a}t Hannover,  Welfengarten 1, 30167 Hannover}
\email{yuguangzhang76@yahoo.com}
\title[Pair-of-pants decompositions of  4-manifolds]{Pair-of-pants decompositions of 4-manifolds diffeomorphic to  general type   hypersurfaces}
\begin{document}
\begin{abstract}
In this paper, we show that a smooth  4-manifold diffeomorphic to  a complex  hypersurface in  $\mathbb{CP}^3 $ of degree $d\geq 5$ can be decomposed as the union of $d(d-4)^2$ copies of 4-dimensional pair-of-pants and certain subsets of  K3 surfaces.
\end{abstract}

\maketitle

\section{Introduction}
A compact Riemann surface $\Sigma$ admits Riemannian  metrics of  constant Gaussian curvature, while
 a 3-dimensional manifold may not
admit any constant curvature metric. Instead  the Perelman's theorem on the Thurston's geometrisation conjecture asserts that a 3-manifold can be  canonically
decomposed  into domains, and some of them carry  complete constant curvature  Riemannian  metrics of finite volume (cf. \cite{P1, KL}). A difference is that
unlike the  3-dimensional case, if the genus of  a Riemann surface  is  bigger than one, the hyperbolic metrics are not unique, and form a moduli space.  However $\Sigma$ admits so called pair-of-pants decompositions which restore certain aspects of the uniqueness (It is a standard topic in topology. See `Pair of pants (mathematics)' in  www.wikipedia.org).

A pair-of-pants
$\mathcal{P}^1$ of real dimension two, or complex dimension one,    is defined as the complement set of three generic points  in $\mathbb{CP}^1$, i.e. $\mathcal{P}^1=\mathbb{CP}^1\backslash \{0, 1, \infty \}$. There is a unique complete Riemannian metric $\mathrm{g}$ on $\mathcal{P}^1$ with Gaussian curvature $-1$ and finite volume ${\rm Vol}_{\mathrm{g}}(\mathcal{P}^1)=2\pi$.  There is a fibration structure  $\mathcal{P}^1 \rightarrow \mathrm{Y}$ from  $\mathcal{P}^1 $ to a graph of  $ \mathrm{Y}$-shape with generic fibres $S^1$, and one singular fibre of shape $\ominus$ over  the vertex.

 If $\Sigma$ is a compact Riemann surface of genus $g\geq 2$, then there is an open subset $ \Sigma^o\subset \Sigma$
consisting exactly of  $2g-2$ copies of pair-of-pants, and the complement $\Sigma\backslash \Sigma^o$ is the disjoint union of $3g-3$ circles. Each circle is not homotopic to a constant curve in $\Sigma$.  Therefore $\Sigma$ is decomposed into canonical local pieces, the pair-of-pants, glued  along circles.   The Gauss-Bonnet
formula expresses  the Euler number via the volume of $\mathrm{g}$ \begin{equation}\label{eq0.1}\chi(\Sigma)=-\frac{1}{2\pi}\sum^{2g-2}{\rm
Vol}_{\mathrm{g}}(\mathcal{P}^1)=\frac{1-g}{\pi}{\rm Vol}_{\mathrm{g}}(\mathcal{P}^1).\end{equation} The combinatoric structure of the decomposition
$\Sigma= \Sigma^o \cup (\Sigma\backslash \Sigma^o)$ is represented by a cubic graph, i.e. any vertex is of shape $\mathrm{Y}$,  where each pair-of-pants corresponds to one vertex, and any edge
associates with  a circle  in  $\Sigma\backslash \Sigma^o$.

\begin{figure}\label{pic0}
{\footnotesize
\caption{Fibration $\mathcal{P}^1 \rightarrow \mathrm{Y}$, and two  graphs associated with pair-of-pants decompositions of Riemann surfaces of genus 2.}}
 \begin{center}
  \setlength{\unitlength}{0.5cm}
 \begin{picture}(8,4)(-2,-2)
 \put(0.1,-0.15){\line(1,0){2}}
  \put(2.12,-1.95){\line(0,1){1.8}}
  \put(2.1,-0.15){\line(1,1){2}}
  \put(2.1,-0.15){\circle*{0.5}}
   \put(-1,-0.15){$\mathrm{Y}$}
\qbezier(0,0)(2,0)(4,2)
\qbezier(0,-0.3)(2,-0.3)(2,-2)
\qbezier(2.3,-2)(2.3,0)(4,1.7)
\qbezier(4,1.7)(4.1,1.8)(4,2)
\qbezier(0,-0.3)(0.15,-0.15)(0,0)
\qbezier(0,-0.3)(-0.15,-0.15)(0,0)
\qbezier(2,-2)(2.15,-1.85)(2.3,-2)
\qbezier(2,-2)(2.15,-2.15)(2.3,-2)
\put(3,-1){$\mathcal{P}^1$}
\qbezier(0,-1.5)(-1,0)(-0.5,-1.5)
\qbezier(0,-1.5)(-1,-2.5)(-0.5,-1.5)
\qbezier(0,-1.5)(1,-2.5)(-0.5,-1.5)
\put(-2.5, -1.5){$\ominus =$}
\end{picture}
  \setlength{\unitlength}{0.5cm}
  \begin{picture}(8,4)(1,1)
\put(2,2){\circle{2}}
\put(6,2){\circle{2}}
\put(10,2){\circle{2}}
\put(3,2){\line(1,0){2}}
\put(9,2){\line(1,0){2}}
\put(3,2){\circle*{0.5}}
\put(5,2){\circle*{0.5}}
\put(9,2){\circle*{0.5}}
\put(11,2){\circle*{0.5}}
\end{picture}
 \end{center}
 \end{figure}
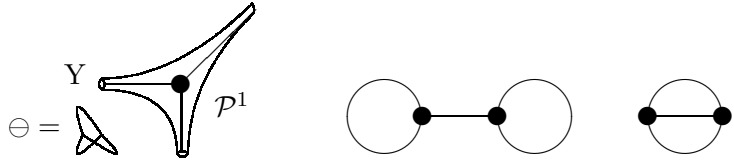

In \cite{Mik}, Mikhalkin has generalised the notion of pair-of-pants to the case of  any even  dimension, and
proved that a smooth complex hypersurface in $\mathbb{CP}^{n+1}$, and more general toric manifolds,  admits pair-of-pants decompositions. \cite{GM} studied pair-of-pants decompositions for real
4-dimensional manifolds from the topology perspective. It is shown in \cite{GM} that every finitely presented group is the fundamental group of a 4-manifold admitting pair-of-pants decompositions. In this paper, we   study some differential/algebraic geometry aspects of
pair-of-pants decompositions for general type hypersurfaces in $\mathbb{CP}^{3}$.

A    pair-of-pants $\mathcal{P}^2$ of real dimension 4, equivalently  complex dimension 2,  is defined as the
complement of 4 general  positioned lines $D$  in  $\mathbb {CP}^2$, i.e. $\mathcal{P}^2=\mathbb{CP}^2\backslash D$ (cf. \cite{Mik}), where $D$ can be
chosen as $$D=\{[z_0, z_1, z_2]\in \mathbb{CP}^2|(z_0+z_1+z_2)z_0z_1z_2=0\}.$$ Equivalently, $$ \mathcal{P}^2=\{(w_1, w_2)\in (\mathbb{C}^*)^2| 1+w_1+w_2\neq0\}.$$ If the compact pair-of-pants is defined as $\bar{\mathcal{P}}^2=\mathbb {CP}^2\backslash \tilde{D}$, where
 $\tilde{D}$ denotes the union of small  tubular open  neighbourhoods  of the 4 generic lines, then $\bar{\mathcal{P}}^2\subset \mathcal{P}^2$, and the interior ${\rm int}(\bar{\mathcal{P}}^2)$ is diffeomorphic to $\mathcal{P}^2$.  The boundary $\partial \bar{\mathcal{P}}^2$ consists of 6 copies of 2-torus $T^2$ and 4 copies of the total space $F$ of the trivial $S^1$-bundle over $\mathcal{P}^1$.  By composing with  the fibration  $\mathcal{P}^1 \rightarrow \mathrm{Y}$,
   $F\rightarrow \mathrm{Y}$ is a fibration  with generic fibres $T^2$, and one singular fibre of shape $\ominus \times \bigcirc $ over the vertex of $\mathrm{Y}$.

As in the case of Riemann surfaces, $\mathcal{P}^2$  carries a natural
complete Einstein metric of  finite volume.  Many works were devoted to generalise Yau and Aubin's theorem (\cite{Yau1, Yau2, Au}) on the Calabi conjecture of projective manifolds with  negative first    Chern class   to the quasi-projective case
   under various conditions (cf. \cite{CY, Ko1, Ko2, TY} etc.). In the
current case, $\mathcal{P}^2$ is a quasi-projective manifold, and the log-canonical divisor $K_{\mathbb{CP}^{2}}+D=H$ is ample, where  $H$ denotes the hyperplane class. Since $D$ has only simple normal crossing singularities,    Theorem 1 in \cite{Ko1} shows that there is a complete K\"{a}hler-Einstein metric $\omega$ on $\mathcal{P}^2$ with  Ricci curvature $-1$ and finite volume $
2\pi^2(K_{\mathbb{CP}^{2}}+D)^2=2\pi^2$.

 \begin{figure}\label{pic2}
{\footnotesize
\caption{4--dimensional pair-of-pants  $\mathcal{P}^2=\mathbb{CP}^2\backslash D$}.}
 \begin{center}
  \setlength{\unitlength}{0.5cm}
  \begin{picture}(8,6)(-1,-1)
\put(-1,0){\line(1,0){7}}
\put(1,-1){\line(0,1){6}}
\put(0,5){\line(1,-1){6}}
\put(-1,-1){\line(1,1){5}}
\put(2,-1){$z_2=0$}
\put(-2,2){$z_1=0$}
\put(4,1){$z_0=0$}
\put(4,3){$z_0+z_1+z_2=0$}
\end{picture}
 \end{center}
 \end{figure}

 If $X$ is a smooth complex hypersurface in $\mathbb{CP}^{3}$ of degree $d\geq 5$, then the canonical divisor
$K_X$ of $X$ is very  ample, and more precisely  the canonical bundle $\mathcal{O}_X(K_X)\cong \mathcal{O}_{\mathbb{CP}^{3}}(d-4)|_X$ by the adjunction formula. $X$ is a minimal algebraic surface of general type. The Yau-Aubin-Calabi theorem asserts that there is a unique K\"{a}hler-Einstein metric with Ricci curvature $-1$ and volume $2\pi^2K_{X_t}^2=2\pi^2d(d-4)^2 $. When we deform the complex structure of $X$, the K\"{a}hler-Einstein metric varies along the deformation. Therefore, if $M$ is a smooth 4-manifold diffeomorphic to $X$, K\"{a}hler-Einstein metrics on $M$ are not unique, and form a    moduli space in  a certain sense, which is a  situation analogous to  the case of Riemann surfaces of positive  genus.

  The
Mikhalkin's theorem in \cite{Mik} asserts that $X$ admits a pair-of-pants decomposition consisting of $d^3$ copies of pair-of-pants, i.e. there is an open dense subset $X^o \subset X$ diffeomorphic to the disjoint union of  $d^3$ copies of pair-of-pants $\mathcal{P}^2$. The number of pair-of-pants
is bigger than that we expect as shown in the following example.

We consider a family of hypersurfaces $X_t$ in $\mathbb{CP}^{3}$ of degree 5 defined by  $$ \sum_{i=0}^3z_i^5+t(z_0+z_1+z_2+z_3)\prod_{j=0}^3z_j= 0, \  \  \ [z_0,z_1,z_2,z_3]\in \mathbb{CP}^{3}, $$ where $t\in [0,\infty)$. When $t\rightarrow\infty$, $X_t$ tends to the  singular
variety $X_0$ given by  $$ (z_0+z_1+z_2+z_3)z_0z_1z_2z_3=0, $$ which is the union of 5 generic hyperplanes in $\mathbb{CP}^{3}$, and the
regular locus $X_{0}^o$ consists of  5 copies of pair-of-pants $\mathcal{P} ^2$.  We can  diffeomorpically embed  $X_{0}^o \hookrightarrow X_t$ for
$t\gg 1$, and obtain a decomposition of $X_t$ with 5 copies of pair-of-pants. And $5<5^3=125$. Furthermore, the number $5$ is  the intersection number $K_{X_t}^2$.

The following
theorem shows that the number of pair-of-pants can be reduced to the expected one, if we allow components  of  other types appearing  in the decomposition.

\begin{theorem}\label{mainthm}
  Let $M$ be a 4-manifold diffeomorphic to a hypersurface $X$ in $\mathbb{CP}^3$ of degree $d\geq 5$, and  $p_g$ be  the geometric genus of $X$.  \begin{enumerate}
\item[i)] There
is an open subset $M^o \subset M$ such that  $M^o$ is diffeomorphic to the disjoint union of  $d(d-4) ^2$ copies
of  pair-of-pants $\mathcal{P}^2$, i.e. $$ M^o=\coprod^{d(d-4)^2}\mathcal{P}^2.$$
   \item[ii)] There is a fibration $\lambda: \overline{M^o} \backslash M^o \rightarrow B$ onto a graph  $B$ with generic fibres   2-torus $T^2$ and finite singular fibres of shape  $\ominus \times \bigcirc $, where $\overline{M^o}$ denotes the closure of $M^o$ in $M$.
\item[iii)]  $M\backslash \overline{M^o}$ is the union of
  subsets of the K3 surface $Y$ including $(\mathbb{C}^*)^2$ and $\mathbb{C}\times \mathbb{C}^*$. Furthermore, there is an open subset  $M' \subset M\backslash
\overline{M^o}$  admitting a smooth  embedding $\iota: M' \hookrightarrow Y$ such that the closure of the image is the K3 surface, i.e.  $\overline{\iota
(M')}=Y$.
  \item [iv)] Any generic fibre $T^2$ of $\lambda$ represents a non-trivial homological  class of  $M$, i.e. $$0\neq [T^2] \in H_2(M, \mathbb{R}).$$ Moreover, if $\varpi$ is a symplectic form representing $c_1(\mathcal{O}_{\mathbb{CP}^3}(1))|_{X} $, then $$\int_{T^2}\varpi =0.$$
  \item [v)]  If $\chi(M)$ is the Euler number of $M$, and $\tau(M)$ denotes the signature of $M$, then   $$   2\chi(M)+3\tau(M)=\frac{d
(d-4)^2}{2\pi^2}{\rm Vol}_\omega(\mathcal{P}^2),$$ where $\omega $ is the complete  K\"{a}hler-Einstein metric  with Ricci curvature   $-1$ on
$\mathcal{P}^2$.
 \item [vi)]
 There is a family $X_t$, $t\in [0, \infty)$, of degree $d$  hypersurfaces  in  $\mathbb{CP}^3$, and there are embeddings  $\Psi_t: X_t \rightarrow \mathbb{CP} ^{p_g-1}$ with $\mathcal{O}_{X_t}(K_{X_t})=\Psi_t^*\mathcal{O}_{\mathbb{CP} ^{p_g-1}}(1)$,
 such that  $\Psi_t(X_t)$ converges to
a singular variety $ X_0$ in $\mathbb{CP}^{p_g-1}$ as analytic subsets, when $t\rightarrow\infty$,  and  the regular locus  of $X_0$ is   diffeomorphic to $M^o$.
\end{enumerate}
\end{theorem}

In this theorem, we also regard  certain subsets of the complex torus $(\mathbb{C}^*)^2$, $\mathbb{C}\times \mathbb{C}^*$,  and the K3
surface as basic building blocks besides the pair-of-pants.
  It has been explored to decompose general
type hypersurfaces into Calabi-Yau components in \cite{LW}.  The pair-of-pants part $M^o$ is an analogue  of hyperbolic pieces, and the complement $M\backslash
\overline{M^o}$ plays a similar  role as graph manifolds in the  decompositions    of 3-manifolds.

 The conclusions ii) and iv) resemble  the  facts that hyperbolic pieces are glued with other parts along incompressible tori  in the 3-dimensional case, and the circles in the pair-of-pants decompositions of Riemann surfaces represent non-trivial classes in   the fundamental group. For example, $\mathbb{CP}^2$ admits a decomposition $\mathbb{CP}^2=
\bar{\mathcal{P}}^2\cup \tilde{D}$,  and the generic fibre  $T^2$ in  $\partial \bar{\mathcal{P}}^2$  represents the  zero class  in $H_2(\mathbb{CP}^2, \mathbb{R})$.   At least in this case, iv) plays a similar  role  as those incompressible tori  in the 3-dimensional decomposition. Unfortunately,  iv) is much weaker than either cases as the homological group is used.

 The assertion v) expresses the Hitchin-Thorpe inequality via the volume of the K\"{a}hler-Einstein metric on
the pair-of-pants, and generalises (\ref{eq0.1}).  vi) shows  that  the expected pair-of-pants   appear in a classical
algebro-geometric way.
 What happens is that the $d(d-4)^2$ copies of pair-of-pants obtained in i)  converge nicely  to the regular locus of $X_0$, while those K3 components are crushed into the singularities of $X_0$.

 We are   motivated by  a paper  \cite{CT}  of Cheeger and Tian about the collapsing of Einstein
4-manifolds, and we recall the relevant  results  in   \cite{CT}. Let $\mathrm{g}_i$ be a sequence of Einstein metrics on a 4-manifold $M$ with Ricci curvature ${\rm Ric}(\mathrm{g}_i)=-\mathrm{g}_i$ and volume ${\rm Vol}_{\mathrm{g}_i}(M)\equiv $const., which
could be obtained by applying    Yau and Aubin's theorem on the Calabi conjecture to hypersurfaces $X$ of degree $d\geq 5$.
Theorem 10.5 of \cite{CT} shows that  by
passing to a subsequence and choosing certain base  points,  $ (M, \mathrm{g}_i)$ converges to   $\coprod\limits_{\nu=1}^{K} (M_\nu, \mathrm{g}_\infty)$, when
$ i\rightarrow \infty,$ in the pointed  Gromov-Hausdorff sense, and $$ {\rm Vol}_{\mathrm{g}_i}(M)=\sum\limits_{\nu=1}^{K} {\rm Vol}_{\mathrm{g}_\infty}
(M_\nu),$$  where each  $M_\nu$ admits at most finite isolated orbifold points as singularities, and $\mathrm{g}_\infty$ are complete negative Einstein metrics in the
orbifold sense.
Consequently,  for $i\gg 1$, $\mathrm{g}_i$ induces a thick-thin decomposition $M=M_0 \cup (M\backslash M_0)$, where the thick part $(M_0,
\mathrm{g}_i)$ is close to   the pointed  Gromov-Hausdorff  limit  $\coprod\limits_{\nu=1}^{K} (M_\nu, \mathrm{g}_\infty)$,  and the thin part $M\backslash M_0$ supports metrics with
bounded Ricci curvature and volume collapsing. Moreover,  $\mathrm{g}_i$ has bounded sectional  curvature on  a subset $M'\subset M\backslash M_0$ by Theorem 0.8 of  \cite{CT}, and therefore $M'$ admits an $F$-structure of positive rank in the Cheeger-Gromov sense  (cf. \cite{CG1,CG2}).

In Theorem \ref{mainthm}, $M^o$ is expected  to be the thick part, the complement $M\backslash \overline{M^o}$ should
be the thin part, and the fibration $\overline{M^o} \backslash M^o \rightarrow B$ is an example of $F$-structure of positive rank.   The limiting behaviour of negative K\"{a}hler-Einstein metrics  along degenerations of  algebraic  manifolds with negative first Chern class   has been studied  under various hypotheses  (cf. \cite{Tian,ruan1,ruan2,ruan3,Zh, SW} and the  references in them), which can certainly be applied to the current situation. Especially, \cite{ruan3} has explored the connection between pair-of-pants degenerations and the convergence of K\"{a}hler-Einstein metrics.

We  outline  the paper briefly.  Section 2 presents  a variant of Theorem \ref{mainthm}, which shows a correlation  between the appearance
of pair-of-pants and the negativity of scalar curvature of metrics allowed on 4-manifolds in the current case. In Section 3, we give some further remarks about the geometry of  pair-of-pants. We  prove the assertions i), ii), iii), and v) of Theorem \ref{mainthm} in Section 4, which  depends on Mikhalkin's work  \cite{Mik} and tropical geometry. Finally,  Section 5 proves the conclusions iv) and vi) of  Theorem \ref{mainthm}.

In this paper,  we use the Riemannian geometric convention of scalar curvature instead of the K\"{a}hler geometry  one, i.e. the
scalar curvature of a metric on a Riemann surface  equals  to  the twice of the Gaussian curvature.

\section{Yamabe invariant}
In the case of Riemann surfaces, (\ref{eq0.1}) shows a strong correlation between the appearance
of pair-of-pants and the negativity of scalar curvature of metrics  allowed on the Riemann surfaces. However this information is lost in v) of Theorem
\ref{mainthm}. The goal of this section is to present a variant of v) in  Theorem \ref{mainthm}, which restores  this   correlation.

First, we recall the definition of the  Yamabe invariant.
Let $M$ be a smooth  compact  real manifold of
dimension $n$. For any conformal class $\mathrm{c}$ on $M$, the Yamabe constant of $\mathrm{c}$ is defined as $$ \sigma(M, \mathrm{c} )=\inf_{\mathrm{g}\in
\mathrm{c}}\big({\rm Vol}_{\mathrm{g}}^{\frac{2-n}{n}}(M)\int_M R_{\mathrm{g}}dv_{\mathrm{g}}\big),$$ where $R_{\mathrm{g}}$ denotes the scalar curvature of
$\mathrm{g}$. The supremum
  $$\sigma(M) = \sup_{\mathrm{c}\in \mathrm{C}}  \sigma(M, \mathrm{c})$$ is a diffeomorphism invariant, and is called the
 Yamabe
invariant of $M$ (cf. \cite{Sc,Ko0}), where $\mathrm{C}$ denotes the set of conformal classes on $M$. If $n=2$, the Gauss-Bonnet formula  implies that the Yamabe invariant  equals
to $2\pi\chi(M) $.

  When  $n=3$, a Perelman's theorem (Section 8.2 in \cite{P1}, and see also Proposition 93.9 in \cite{KL}) answers a conjecture of  Anderson (cf.  \cite{An}), and  says that if $\sigma(M)<0$,
then there is an open subset $M_0 \subset M$ admitting a complete hyperbolic metric $\mathrm{g}$ with sectional curvature $-1$, and the Yamabe invariant  is
realised by the volume of $\mathrm{g}$, i.e.  $$ \sigma(M) =-6{\rm Vol}_{\mathrm{g}} ^{\frac{2}{3}}(M_0).$$ Actually,  Perelman   used the  $\bar{\lambda}$-invariant introduced in \cite{P}, which was proved to  equal to the Yamabe invariant when $\sigma(M) \leq 0$ by \cite{AIL}.

 In the case of dimension $4$, LeBrun proved
an analogous equality in \cite{le1, le2} for manifolds diffeomorphic to minimal K\"{a}hler surfaces of  non-negative Kodaira dimension.  We take this
opportunity to give  an alternative proof of the case of  general type surfaces.

\begin{theorem}[Theorem 7 in \cite{le1}]\label{thm3}
  Let $M$ be  a smooth  4-manifold diffeomorphic to a
minimal projective  surface $X$ of  general type, i.e. the canonical divisor $K_X$ of $X$ is nef and big $K_X^2>0$.   Then the Yamabe invariant   $$\sigma(M) =-\sqrt{32 \pi^2 K_X^2}= -4\sqrt{{\rm
Vol}_\omega(X_0)}, $$ where
 $\omega$ is a K\"{a}hler-Einstein metric  with Ricci curvature   $-1$ on an open
subset  $X_0 \subset X$.
\end{theorem}

\begin{proof}[An alternative proof]   We estimate the upper
bound and the lower bound of the Yamabe invariant $\sigma(M)$, and show that they match. The upper bound  follows LeBrun's argument  via the Seiberg-Witten
equation with a minor twist, and we present it here for the completeness. The difference is to prove the lower bound where we use the K\"{a}hler-Ricci flow
instead of the original approach in \cite{le1}.

First, we recall the reinterpretation of the Yamabe invariant due to Lott and Kleiner \cite{KL}. For any 4-manifold $M$,
(93.6) in \cite{KL} asserts that if
$\sigma(M)\leq 0$, then  \begin{equation}\label{eqKL} \sigma(M) = \sup_{\mathrm{g}\in \mathcal{M}}
\check{R}(\mathrm{g})\sqrt{{\rm Vol}_{\mathrm{g}}(M)},  \end{equation} where  $\check{R}(\mathrm{g})$ denotes the minimal of the scalar curvature of
$\mathrm{g}$, i.e. $$ \check{R}(\mathrm{g }) = \inf_{x \in M} R(\mathrm{g})(x).$$

  If $\mathfrak{c}$ is a Spin$^c$ structure on $M$, and
$S_{\mathfrak{c}}^{\pm}$ denotes the Spin$^c$   bundle, then the Seiberg-Witten equation is introduced in \cite{Witt}, and reads  $$\mathfrak{D}_A \phi =0, \ \
\   \  \ F_A^+=\mathfrak{q}( \phi ),$$ for an unknown positive spinor $\phi$ and an unknown  $U(1)$-connection $A$ on the determinant bundle $\mathfrak{L}$ of
$\mathfrak{c}$, where $\mathfrak{D}_A$ denotes the Dirac operator, $F_A$ is the curvature of $A$,  and $ \mathfrak{q}( \phi )$ is a   quadratic form of
$\phi$ satisfying $|\mathfrak{q}( \phi )|^2=\frac{1}{8}|\phi|^4 $.   The Seiberg-Witten invariant $SW_M(\mathfrak{c})$ is a diffeomorphism invariant defined
via the moduli space of the solutions $(\phi, A)$ module $U(1)$-gauge changes (cf. \cite{Witt, Mo}).  For example,  if $M$ is diffeomorphic to a minimal  K\"{a}hler surface
$X$ of general type,   and $\mathfrak{L}$ is the anti-canonical bundle $\mathfrak{L}=\mathcal{O}_X(K_X^{-1}) $, then $SW_M(\mathfrak{c})\neq 0$ by \cite{le0}, which implies
that for any Riemannian  metric $\mathrm{g}$, the Seiberg-Witten equation  has a  solution $(\phi, A)$. Since  $c_1^2(\mathfrak{L})=K_X^2>0$, $F_A^+$ and $\phi$ are not identical to zero.

   The Weitzenb\"{o}ck formula  says $$4\nabla_A^*\nabla_A \phi
+  R(\mathrm{g})\phi + |\phi|^2\phi=0.$$ By  producing  with $\phi$ and integration, we have $$ \check{R}(\mathrm{g})\int_M |\phi|^2dv_{\mathrm{g} } \leq
\int_M (4|\nabla_A \phi |^2+R(\mathrm{g})|\phi|^2)dv_{\mathrm{g}}= - \int_M |\phi|^4dv_{\mathrm{g}},$$ and by the Schwarz inequality, $$
\check{R}(\mathrm{g})\sqrt{
{\rm Vol}_{\mathrm {g}}(M)} \leq -\sqrt{ \int_M |\phi|^4dv_{\mathrm{g}}}=-\sqrt{ 8\int_M |F_A^+|^2dv_{\mathrm{g}}}\leq
-\sqrt{32\pi^2c_1^2(\mathfrak{L})}.$$ Therefore, $M$ does not admit any Riemannian metric of  positive scalar curvature, and the  Yamabe invariant $\sigma(M)\leq 0$.  We obtain  the upper bound $$\sigma(M) \leq - \sqrt{32\pi^2K_X^2},$$ by  taking the supremum and (\ref{eqKL})
  (See \cite{FZ} for another
proof via the Perelman's $\bar{\lambda}$-functional).

Now we prove the lower bound by  considering  the K\"{a}hler-Ricci flow
\begin{equation}\label{RF}\frac{\partial \omega_t}{\partial t} =-{\rm Ric}(\omega_t)-\omega_t,  \ \  \ \ t\in [0, T), \end{equation} on $X$ with an initial K\"{a}hler  metric $\omega_0$.
This  version of  K\"{a}hler-Ricci flow was intensely studied in recent years for   the
differential geometric understanding of the minimal model programme   (cf. \cite{Ts,TZ,ST,To} and the references in them). See also \cite{FZZ1,FZZ2} for some interactions between the real Ricci flow and the Seiberg-Witten equation.

Let $X_{can}$ be the canonical model of $X$, and $\pi: X \rightarrow X_{can}$ be the contraction map. Note that
$X_{can}$ is a 2-dimensional normal  variety with only finite A-D-E singularities, and  $-c_1(X)=c_1(\mathcal{O}_X(K_X) )=\pi^*\alpha$ for an ample class $\alpha$ on $X_{can}$.
  In \cite{Ts, TZ}, it is proved that the solution
$\omega_t$ exists for a long time $t\in [0, \infty)$, i.e. $T=\infty$, and
   $\omega_t$ converges to a semi-positive
current $\omega$ presenting $-2\pi c_1(X)$  with bounded local potential functions
   when $t\rightarrow \infty$. Furthermore, $\omega$ is a smooth
K\"{a}hler-Einstein metric with Ricci curvature $-1$ on $\pi^{-1}(X_{can}^o)$, where $X_{can}^o $ is the regular locus of $X_{can}$, and $$2\pi^2K_X^2=2\pi^2c_1^2(X)={\rm Vol}_\omega (X_{can}^o). $$

The evolution equation of scalar curvature is $$\frac{\partial R_t}{\partial
t} = \Delta_{t} R_t + 2|{\rm Ric}_ t|^2 + R_t=\Delta_{t} R_t + 2|{\rm Ric}_t^o|^2 -( R_t+4)$$   (cf. Lemma 2.38 in
\cite{C-N}),  where $R_t=R(\omega_t)$, ${\rm Ric}_t={\rm Ric}(\omega_t)$ and
${\rm Ric}_t^o={\rm Ric}_t+\omega_t$, which satisfies $|{\rm Ric}_t^o |^2 =|{\rm Ric}_t|^2+R_t+2$  by $R_t=2{\rm tr}_{\omega_t}{\rm Ri c}_t$.
By the maximal principle, the minimal $\check{R}_t$ of $R_t$ satisfies $$\frac{\partial \check{R}_t}{\partial t} \geq - (\check{R}_t+4),$$ and
therefore  we obtain $$\check{R}_t\geq -4 -Ce^{-t}\rightarrow -4,$$ for a constant $C>0$ independent of $t$.

Since the Chern-Weil theory shows that the
cohomological  classes evolve   as $$\frac{\partial [\omega_t]}{\partial t}=-2\pi c_1(X)-[\omega_t], $$ we solve the ordinary differential equation  and obtain  $$ [\omega_t]=-2\pi
c_1(X)+e^{-t}(2\pi c_1(X )+[\omega_0]).$$ Thus the volumes $${\rm Vol}_{\omega_t}(X)=\frac{1}{2}\int_X \omega_t^2 \rightarrow 2\pi^2 c_1^2(X)= 2\pi^2K_X^2,$$ when $t\rightarrow \infty$. We obtain  the  lower bound of the Yamabe invariant $$\sigma(M) \geq \lim\limits_{t\rightarrow\infty}\check{R}_ t\sqrt{{\rm
Vol}_{\omega_t}(X)}= -\sqrt{32\pi^2 K_X^2},$$ which is the same as the upper bound. We
obtain the conclusion  by letting   $X_0=\pi^{-1}(X_{can}^o)$.
\end{proof}

The Riemann-Roch theorem  asserts $$K_X^2=2\chi(X)+3\tau(X)$$ for any compact  complex surface $X$. The canonical divisor
$K_X$ of a hypersurface $X$ in $\mathbb{CP}^3$ of degree $d\geq 5$ is ample, and thus $X$ is of minimal general type. We reach a variant of  Theorem \ref{mainthm}
by using Theorem \ref{thm3}, which shows a correlation between the negativity of the  Yamabe invariant and the appearance of pair-of-pants.

\begin{theorem}\label{thm+}
  Under the setup of Theorem \ref{mainthm},  the Yamabe invariant of $M$ is    $$\sigma(M)= - 4\sqrt{d(d-4)^2{\rm
Vol}_\omega(\mathcal{P}^2)},$$ where
 $\omega$ is the K\"{a}hler-Einstein metric  with Ricci curvature   $- 1$ on the pair-of-pants
$\mathcal{P}^2$. Equivalently, the number of pair-of-pants in $M^o$    equals to $$\frac{1}{32\pi^2}(\max \{0, -\sigma(M)\})^2.$$
\end{theorem}

Finally, we remark that the assertions  ii) and iv) in Theorem \ref{mainthm}  provide  a certain constraint on  the Yamabe invariant. The proof can be applied to more broader  scenarios, and therefore  we present a general result, which might have some independent interests.

\begin{theorem}\label{thm+0000}
  Let $M$ be a compact  4-dimensional symplectic manifold, and $\varpi$ be a symplectic form on $M$. If there is a  Lagrangian submanifold  $ \Sigma$  in $M$ such that $\Sigma$ is a Riemann surface of genus $g\geq 1$, and
   represents   a non-trivial homological  class in $M$, i.e. $$ 0\neq [\Sigma]\in H_2(M, \mathbb{R}),$$ then $M$ does not admit any Riemannian metric of  positive scalar curvature. Furthermore, if $M$ is  minimal, then  $2\chi(M)+3\tau(M)\geq0$, and   the Yamabe invariant \begin{equation}\label{RF000}\sigma(M)\leq -\sqrt{32\pi^2(2\chi(M)+3\tau(M))}.\end{equation}  If $M$ is a finite covering of  an oriented manifold $\bar{M}$,  then (\ref{RF000}) also  holds for $\bar{M}$.  \end{theorem}

\begin{proof} The Weinstein neighbourhood theorem (cf. \cite{Wei})  says that there is a tubular  neighbourhood of $\Sigma$ in $M$ diffeomorphic to a  neighbourhood of the zero section in the cotangent bundle $T^*\Sigma$. Since the first Chern number of $T^*\Sigma$ is the minus of the Euler number of $\Sigma$, i.e.  $\int_{\Sigma}c_1(T^*\Sigma)=-\chi(\Sigma)$, the self-intersection number $\Sigma\cdot \Sigma=2g-2\geq 0$.

Let $\mathrm{g}$ be a Riemannian metric  compatible with $\varpi$, i.e. $\mathrm{g}(\cdot, \cdot)=\varpi(\cdot, J \cdot)$ for an almost complex structure $J$ compatible with $\varpi$.  Then $\varpi$  is a self-dual 2-form with respect to $\mathrm{g}$, i.e. $\ast \varpi =\varpi$, where $\ast$ denotes the Hodge star operator of $\mathrm{g}$. We identify $H^2(M, \mathbb{R})$ with the sum of spaces of self-dual and anti-self-dual harmonic 2-forms via the Hodge theory, i.e. $H^2(M, \mathbb{R})=\mathcal{H}_{+}(M)\oplus \mathcal{H}_{-}(M)$.   If $A\in H^2(M, \mathbb{R})$ is the Poincar\'{e} dual of $[\Sigma]$, then $A=\alpha +\beta$ such that $d\alpha=d\beta=0$,  $\ast \alpha =\alpha$, $\ast \beta=-\beta$, and  $$0\leq \Sigma\cdot \Sigma=\int_M(\alpha +\beta)^2=\int_M \alpha^2+ \int_M \beta^2.$$ Thus $$ \int_M \alpha^2 \geq - \int_M \beta^2 = \int_M \beta \wedge\ast\beta \geq 0.$$
Note that if $\alpha =0$, then $\beta=0$ and $A=0$, which is a contradiction. Thus $\alpha\neq 0$, and $$\int_{\Sigma}\alpha=\int_M\alpha \wedge (\alpha+\beta)=\int_M \alpha^2 \neq 0.$$ Since $$\int_{\Sigma}\varpi =0,$$ $\varpi$ and $\alpha$ are linearly independent in the space of self-dual harmonic  2-forms $\mathcal{H}_{+}(M)$. Therefore, $b_2^+\geq 2$.

Taubes's theorem (cf. \cite{Tau0,Tau,Kot})  asserts that the Seiberg-Witten invariant is non-zero, i.e. $SW_M(\mathfrak{c})\neq 0$,  for a certain Spin$^c$ structure $\mathfrak{c}$.  The same argument as in the proof of Theorem \ref{thm3}   proves that there does not exist any Riemannian metric of positive curvature on $M$, and $\sigma(M)\leq 0$. Furthermore,    if $M$ is minimal, then $\sigma(M)\leq -\sqrt{32\pi^2c_1^2(\mathfrak{L})}$, since
$c_1^2(\mathfrak{L})=2\chi(M)+3\tau(M)\geq 0$ by Theorem 4.11 and  Corollary 4.9 in \cite{Kot} (cf.  \cite{Tau1}), where $\mathfrak{L}$ is the determinant bundle of the Spin$^c$ structure $\mathfrak{c}$.

  Finally,
if $M \rightarrow \bar{M}$ is a finite $\nu$-sheets covering, then $\sqrt{\nu} \sigma(\bar{M})\leq \sigma(M)$,  $\nu\chi(\bar{M})=\chi(M)$,  and $\nu\tau(\bar{M})=\tau(M)$. We obtain the conclusion.
\end{proof}

Note that if we replace the Lagrangian condition in this theorem by the existence of a local $T^2$-fibration satisfying iv) in Theorem \ref{mainthm}, then the self-intersection number $T^2\cdot T^2=0 $, and  the same  argument proves  the same conclusion.  Certain submanifolds provide obstructions for the existence of Riemannian  metrics of positive scalar curvature   by Schoen-Yau \cite{SY} and Gromov-Lawson \cite{GL}. Unlike these earlier pioneer works, the   obstruction in Theorem  \ref{thm+0000} is  the constraint  provided by the Seiberg-Witten theory. Of course,  we should ideally   use the techniques as in \cite{SY,GL} to prove that the existence of pair-of-pants under  some  topological hypotheses    implies the negativity of the Yamabe invariant.

\section{Remarks on pair-of-pants}
In this section, we give some remarks on pair-of-pants, which are not directly used in the proof of the main theorem, but provide a better understanding of the geometry of pair-of-pants.

First, we recall the existence of complete K\"{a}hler-Einstein metrics with negative Ricci curvature  on  quasi-projective manifolds, which was studied   by many authors under various  hypotheses   (cf. \cite{CY, Ko1, Ko2, TY} etc.).
 Let $(X, D)$ be a pair of log K\"{a}hler surface such that the divisor  $D$ has only simple normal crossing singularities and the log canonical divisor $K_X+D$ is ample. Then there is a  complete  K\"{a}hler-Einstein metric $\omega$ on $X\backslash D$  with  finite volume $${\rm Vol}_{\omega}(X\backslash D)= 2\pi^2(K_X+D)^2, \ \ \ \ {\rm and} \ \ {\rm Ric}(\omega)=-\omega,$$ by Theorem 1 in \cite{Ko1}.  Since $(K_X+D)^2\in\mathbb{Z}$,  the smallest possible value of the volumes of complete  K\"{a}hler-Einstein metrics constructed  by this method is  $2\pi^2$, i.e. $(K_X+D)^2=1$.

We apply this result to $(\mathbb{CP}^2, D)$ with $D$ being the sum of  $4$ general positioned  lines, i.e. $\mathbb{CP}^2\backslash D =\mathcal{P}^2$,  and obtain a complete K\"{a}hler-Einstein metric $\omega$ on the 4-dimensional pair-of-pants $\mathcal{P}^2$ with  \begin{equation}\label{eq3.1}{\rm Ric}(\omega)=-\omega,  \ \ \ {\rm and} \ \  {\rm Vol}_{\omega}(\mathcal{P}^2 )= 2\pi^2,\end{equation} since $K_{\mathbb{CP}^2}=-3H$ and $D=4H$, where $H$ denotes  the hyperplane class.  Note that complete K\"{a}hler-Einstein metrics with the numerical  data (\ref{eq3.1}) are certainly not unique. For example, if $D'$ is a smooth curve of degree 4, then there is a complete K\"{a}hler-Einstein metric on $\mathbb{CP}^2\backslash D' $ satisfying  (\ref{eq3.1}), and $\mathbb{CP}^2\backslash D' $ is not diffeomorphic to $\mathcal{P}^2$. However, the pair-of-pants is still special as  the sum of  $4$ generic lines is the most degenerated reduced curve of degree 4.

We  provide
 a criterion for 4-dimensional   pair-of-pants via the numerical data  (\ref{eq3.1}) in the following proposition.

\begin{proposition}\label{pro1}
Let $X$ be a   smooth minimal  projective surface with vanishing first Betti number, i.e. $b_1(X)=0$,  and $D$ be a reduced effective   divisor  on $X$ with only simple normal crossing singularities.
  Assume that log canonical divisor  $K_X+D$ is ample,  and $D$ is a most degenerated divisor in the sense
that the number of irreducible components of $D$ is bigger or equal to the number of irreducible components of any reduced  divisor $D'$ linearly equivalent to $D$.  Then  $$(K_X+D)^2=1,$$ if and only if $X=\mathbb{CP}^2$ and   $X\backslash D$ is the   pair-of-pants, i.e. $X\backslash D =\mathcal{P}^2$.
\end{proposition}

\begin{proof}
We only need to prove that $(K_X+D)^2=1$ implies  $X=\mathbb{CP}^2$ and $X\backslash D =\mathcal{P}^2$.

Note that $(K_X+D)\cdot D \in\mathbb{Z}$, $(K_X+D)\cdot K_X \in\mathbb{Z}$ and  $$ (K_X+D)^2= (K_X+D)\cdot D + (K_X+D)\cdot K_X =1.$$
By the adjunction formula for singular curves in surfaces, $(K_X+D)\cdot D = 2(g(D)-1), $ where $g(D)$ is the virtual  genus of $D$. Since $K_X+D$ is ample, $(K_X+D)\cdot D >0$,  which implies  $(K_X+D)\cdot D \geq 2$, and   $$(K_X+D)\cdot K_X <0.$$ Hence $K_X$ is not nef, and $X$ is either $\mathbb{CP}^2$ or a minimal ruled surface over a curve by the classification of minimal surfaces (cf.  Theorem 4 in Chapter 10 of  \cite{Fr1}).

Assume that $X$ is a ruled surface over a curve $C$.  Since the first Betti number of $X$ is zero, $C=\mathbb{CP}^1$ (cf. Lemma 13 in Chapter 5 of \cite{Fr1}),  and $X$ is a  Hirzebruch surface, i.e. $X=\mathbb{P}(\mathcal{O}_{\mathbb{CP}^1}\oplus \mathcal{O}_{\mathbb{CP}^1}(-\mu))$.
 Let $s$ be the  0-section with $s^2=-\mu$,  $\mu\geq 0$,  and $l$ be the fibre class which satisfies   $l^2=0$ and $s\cdot l=1$. If $K_X+D = a s +bl$ for $a\in \mathbb{Z}$ and $b\in \mathbb{Z}$, then $$ (K_X+D)\cdot l= a >0,  \  \  \   (K_X+D)\cdot s= -\mu a + b>0,$$ by the ampleness. Thus $$ 1= (K_X+D)^2=-\mu a^2+2ab> \mu a^2\geq 0, $$  and
 $\mu =0$. We obtain  a contradiction  $1=2ab>0$.

Therefore,  $X=\mathbb{CP}^2$, which implies $K_X=-3H$ and $D=4H$, where $H$ is the hyperplane class. We obtain the conclusion since  the sum of  $4$ generic lines is the most degenerated  curve of degree 4.
\end{proof}

Orbifold singularities, especially the A-D-E singularities, appear naturally in the study of minimal models for algebraic  surfaces, and  limit spaces of Einstein 4-manifolds in the Gromov-Hausdorff convergence (cf. \cite{An0,Na}). If singularities present, we certainly have examples of complete  K\"{a}hler-Einstein metrics  in the orbifold sense with volume less than $2\pi^2$.

\begin{example}\label{eg3} We consider the $\mathbb{Z}_3$-action on $\mathbb{CP}^2$ by permutating the  homogeneous coordinates, i.e. $\gamma \cdot [z_0,z_1,z_2]=[z_2,z_0,z_1]$  for the generator $\gamma$ of $\mathbb{Z}_3$.  There are three fixed points $[1,1,1], [1, \varepsilon^2, \varepsilon], [1, \varepsilon, \varepsilon^2]$ where $\varepsilon=\exp \frac{2\pi\sqrt{-1}}{3}$, and the  volume form $\frac{dw_1\wedge dw_2}{w_1w_2} $ is preserved  where $w_1=z_1z_0^{-1}$ and $w_2=z_2z_0^{-1}$.  Note that  $D=\{[z_0,z_1,z_2]\in \mathbb{CP}^2| z_0z_1z_2(z_0+z_1+z_2)=0\}$ is $\mathbb{Z}_3$-invariant, and $ [1, \varepsilon^2, \varepsilon], [1, \varepsilon, \varepsilon^2]\in D$ by $1+ \varepsilon + \varepsilon^2=0 $.
Thus $\mathbb{Z}_3$ acts on the pair-of-pants $\mathcal{P}^2$ with only one fixed point $[1,1,1]$, and the quotient $\mathcal{P}^2/\mathbb{Z}_3$ is an orbifold with  one isolated  singularity. Note that the class $c_1(\mathcal{O}_{\mathbb{CP}^2}(K_{\mathbb{CP}^2}+D))$ is also invariant, and the complete K\"{a}hler-Einstein metric $\omega$ on $\mathcal{P}^2$ descends to a complete  orbifold metric on  $\mathcal{P}^2/\mathbb{Z}_3$, denoted still as $\omega$,  which satisfies $${\rm Ric}(\omega)=-\omega,  \  \  \ {\rm and} \  \  {\rm Vol}_{\omega}(\mathcal{P}^2/\mathbb{Z}_3)= \frac{2\pi^2}{3}.$$
\end{example}

We remark that the K\"{a}hler-Einstein metric $\omega$ on $\mathcal{P}^2$ is not complex hyperbolic as follows.  Note that the Euler number of $\mathcal{P}^2$ is one (cf. Proposition 2.5 in \cite{GM}), and the Gauss-Bonnet-Chern formula still holds in the current case  (cf. Theorem 4.5 in \cite{CT}), i.e. $$1=\chi(\mathcal{P}^2)= \frac{1}{8\pi^2}\int_{\mathcal{P}^2}(\frac{R^2}{24}+|W^+|^2+|W^-|^2)dv_\omega,$$ where $W^\pm$ denotes the self-dual/anti-self-dual Weyl curvature of $\omega$ (cf. \cite{B}). Since $\omega$ is K\"{a}hler, $24|W^+|^2=R^2$, and we obtain \begin{equation}\label{weyl} \int_{\mathcal{P}^2}|W^-|^2dv_\omega =\frac{16}{3} \pi^2.\end{equation} Hence $W^-$ is not identical to zero, and $\omega$  is not a complex hyperbolic metric.
Furthermore, (\ref{weyl}) indicates that  we might not expect that   complex hyperbolic manifolds  admit  sensible pair-of-pants decompositions, but we would   rather view them  as another basic negative pieces because of the Mostow rigidity.

 Many works have been done towards  certain   decompositions of smooth 4-manifolds and canonical metrics   from various perspectives, for example, metric geometry, topology, algebraic geometry and symplectic geometry etc. See \cite{AHS,An1,An2,T0,T,TV,D,FS,Cat,le3} and the references in them, and especially \cite{An2,T,D,FS} for some expository discussions. In this paper, we explore decompositions for  4-manifolds of a very specific type, i.e. projective  hypersurfaces of general type,  while   keeping  in mind the general  questions whether there is  a geometrisation  theory for 4-manifolds, and if it exists at all, what  the zoo of fundamental 4-manifolds might consists of.

Next, we remark that the product of two 2-dimensional pair-of-pants, $ \mathcal{P}^1\times \mathcal{P}^1$, deforms to a singular variety containing  two copies  of   pair-of-pants $\mathcal{P}^2$. Therefore, we may not  regard $ \mathcal{P}^1\times \mathcal{P}^1$ as a fundamental building block.

 \begin{example}\label{ep1}
  There is a  degeneration $f: \mathcal{X}\rightarrow\mathbb{C}$ satisfying  that  $\mathcal{X}$ is smooth, $$(f^{-1}(0))^o=\mathcal{P}^2\coprod \mathcal{P}^2,  \  \  {\rm and} \ \ f^{-1}(1)=\mathcal{P}^1\times \mathcal{P}^1,$$  where $(f^{-1}(0))^o$ denotes  the regular locus of $f^{-1}(0)$.
\end{example}

\begin{proof}
 Let  $$ \bar{\mathcal{X}}=\{(w_0,w_1,w_2,w_3)\in (\mathbb{C})^3\times \mathbb{C}^*| w_0w_3=w_1w_2\},$$   and $f:  \bar{\mathcal{X}} \rightarrow \mathbb{C}$ be given by $(w_0,w_1,w_2,w_3)\mapsto w_0$.
  Note that $ \bar{\mathcal{X}}$ is smooth since the defining equation defines one ordinary double point that is not in $ (\mathbb{C})^3\times \mathbb{C}^*$.  If  $X_{w_0}=f^{-1}(w_0)$, then $X_0$ is a singular variety  given by $w_1w_2=0 $, i.e.   $$X_0=\{(w_1,w_3)\in \mathbb{C}\times \mathbb{C}^*\}\cup \{(w_2,w_3)\in \mathbb{C}\times\mathbb{C}^*\}\subset \mathbb{C}^2\times\mathbb{C}^*, \  \ {\rm and} $$  $$ X_1=\{(w_1,w_2,w_3)\in  \mathbb{C}^2\times \mathbb{C}^*| w_3=w_1w_2\neq 0\}=\{(w_1,w_2)\in (\mathbb{C}^*)^2\}.$$

If we let   $$\mathcal{D}=\{(w_0,w_1,w_2,w_3)\in \bar{\mathcal{X}}| 1+w_1+w_2+w_3=0\},$$  then the regular locus of  $X_0\backslash \mathcal{D}$ is the  Zariski open subset  $$\{(w_1,w_3)\in (\mathbb{C}^*)^2|1+w_1+w_3\neq 0\}\cup \{(w_2,w_3)\in (\mathbb{C}^*)^2|1+w_2+w_3\neq0\},$$
 i.e.  the regular locus of $X_0\backslash \mathcal{D}$ consists of two copies of  pair-of-pants $\mathcal{P}^2$.  By $1+w_1+w_2+w_1w_2=(1+w_1)(1+w_2)$, $$  X_1\backslash \mathcal{D}=\{(w_1,w_2)\in (\mathbb{C}^*)^2|w_1 \neq -1,  w_2 \neq -1\}=\mathcal{P}^1\times \mathcal{P}^1.$$ We obtain the conclusion by setting $\mathcal{X} =\bar{\mathcal{X}}\backslash\mathcal{D}$.
 \end{proof}

 If $D$ is a divisor in $\mathbb{CP}^2$ consisting of   $d\geq 5$ generic positioned lines, then $(K_{\mathbb{CP}^2}+D)^2=(d-3)^2$, and we expect that  $\mathbb{CP}^2\backslash D$ deforms to $(d-3)^2$-copies of pair-of-pants $\mathcal{P}^2 $. Thus we might  not obtain more sensible building blocks by removing lines from the pair-of-pants. $D$ is an example of  arrangements of lines, and one technique to construct  general type surfaces is to use  branched coverings of $\mathbb{CP}^2$ along  arrangements of lines (cf. \cite{Hir}).      We work out the concrete case of $d=5$, and the proof would become  more transparent by using the techniques from   toric  geometry.

\begin{example}\label{ep2} We assume that $$D=\{[z_0, z_1, z_2]\in \mathbb{CP}^2|z_0 z_1z_2(z_0+ z_1+ z_2)(a_0z_0+ a_1z_1+a_2z_2)=0\}$$ for certain generic chosen  $a_0, a_1, a_2\in \mathbb{C}$.   Then there is a degeneration $f: \mathcal{X} \rightarrow \mathbb{C}$ such that    $$(f^{-1}(0))^o=\coprod^4 \mathcal{P}^2,  \  \  {\rm and} \ \ f^{-1}(1)=\mathbb{CP}^2\backslash D,$$  where $(f^{-1}(0))^o$ denotes  the regular locus of $f^{-1}(0)$.
\end{example}

\begin{proof}  We consider  a family of Veronese embeddings $\Psi_w: \mathbb{CP}^2 \hookrightarrow \mathbb{CP}^5$ given by $$Z_0=wz_0^2, \  Z_1=z_0z_1, \ Z_2=z_0z_2, \ Z_3=z_1z_2, \ Z_4=wz_1^2, \ Z_5=wz_2^2,  $$ where $w\in \mathbb{C}^*$, and $[Z_0, Z_1, Z_2, Z_3, Z_4, Z_5]$ are homogeneous coordinates of  $\mathbb{CP}^5$.  The image is given by the equations $$Z_0Z_4=w^2Z_1^2, \ \  Z_0Z_5=w^2Z_2^2, \ \ Z_4Z_5=w^2Z_3^2,  $$ $$wZ_1Z_2=Z_0Z_3, \ \ wZ_2Z_3=Z_1Z_5, \ \ wZ_1Z_3=Z_2Z_4.$$ When $w\rightarrow 0$, $\Psi_w(\mathbb{CP}^2)$ converges to a singular variety $X_0$ with 4 irreducible components, i.e. $X_0=X_0^1\cup X_0^2\cup X_0^3 \cup X_0^4$, where $$X_0^1=\{[Z_0,Z_1, Z_2, 0, 0,0]\in\mathbb{CP}^5\}, \ X_0^2=\{[0,Z_1, Z_2, Z_3, 0,0]\in\mathbb{CP}^5\},$$ $$X_0^3=\{[0,Z_1, 0, Z_3, Z_4,0]\in\mathbb{CP}^5\}, \ X_0^4=\{[0,0, Z_2, Z_3, 0,Z_5]\in\mathbb{CP}^5\}.$$

\begin{figure}\label{pic-100}
{\footnotesize
\caption{Intersection complex of $X_0$ in Example \ref{ep2}. } }
 \begin{center}
 \setlength{\unitlength}{0.5cm}
  \begin{picture}(5,5)(0,0)
  \put(0,0){$Z_0$}
   \put(2,0){$Z_1$}
    \put(0,2){$Z_2$}
    \put(4,0){$Z_4$}
     \put(0,4){$Z_5$}
      \put(3,2.5){$Z_3$}
\put(1,1){\line(1,0){3}}
\put(1,1){\line(0,1){3}}
\put(4,1){\line(-1,1){3}}
\put(1,2.5){\line(1,0){1.5}}
\put(2.5,1){\line(0,1){1.5}}
\put(2.5,1){\line(-1,1){1.5}}
\end{picture}
 \end{center}
 \end{figure}
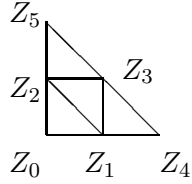

We let 
$D^0=\{[ 0, 0,0, Z_3, Z_4, Z_5]\in\mathbb{CP}^5\}$, $ D^1=\{[ Z_0, 0,Z_2, 0, 0, Z_5]\in\mathbb{CP}^5\}$, $ D^2=\{[ Z_0, Z_1,0, 0, Z_4, 0]\in\mathbb{CP}^5\}$,
and  define  $D'\subset \mathbb{CP}^5$   by $$ a_0Z_0+(a_0+a_1)Z_1+(a_0+a_2)Z_2+(a_1+a_2)Z_3+a_1Z_4+a_2Z_5=0. $$
 Note that $D'\cap X_0^i$, $i=1,2,3,4$, is a  generic line in $\mathbb{CP}^2$. If we regard $\mathbb{CP}^2$ as a toric manifold and   $(\mathbb{C}^*)^2\subset \mathbb{CP}^2$, then  $(\mathbb{C}^*)^2\backslash (D'\cap X_0^i)$ is a pair-of-pants $\mathcal{P}^2$, and does not intersect with any $D^j$, $j=0,1,2$.
  Furthermore,   $$ \Psi_1(D)=\Psi_1(\mathbb{CP}^2)\cap (D'\cup D^1 \cup D^3 \cup D^4) .$$
  We obtain the conclusion by letting $$\mathcal{X}=(X_0\bigcup\limits_{w\in\mathbb{C}^*}\Psi_w(\mathbb{CP}^2))\backslash (D'\cup D^1 \cup D^3 \cup D^4)\subset \mathbb{CP}^5\times \mathbb{C}, $$ and $f$ be the projection to $w$.
\end{proof}

We finish this section by showing that there is a pair-of-pants in the original  Godeaux surface.

\begin{example}\label{ep3} We consider the
Fermat quintic $$X=\{[z_0, z_1, z_2, z_3]\in  \mathbb{CP}^3 |z_0^5+z_1^5+z_2^5+z_3^5=0\},$$ which is a smooth minimal surface of general type with $K_X^2=5$. If  $\mathbb{Z}_5$ acts on $\mathbb{CP}^3$ by $\varepsilon \cdot [z_0, z_1, z_2, z_3]=[z_0, \varepsilon z_1, \varepsilon^2 z_2, \varepsilon^3 z_3]$ where $\varepsilon = \exp\frac{2\pi \sqrt{-1}}{5}$, then $\mathbb{Z}_5$ acts on $X$ freely. The quotient $\tilde{X}=X/\mathbb{Z}_5$ is a smooth surface of general type  with  $K_{\tilde{X}}^2=1$,  called a Godeaux surface (cf. \cite{Pi}). The geometric genus $p_g(\tilde{X})=0$ and the fundamental group $\pi_1( \tilde{X})=\mathbb{Z}_5$.  \end{example}

\begin{theorem}\label{pro2}
If $M$ is diffeomorphic to the Godeaux surface $\tilde{X}$, then there is an open subset $M^o \subset M$ diffeomorphic to the pair-of-pants $\mathcal{P}^2$, and the Yamabe invariant $$\sigma (M)=-4\sqrt{{\rm Vol}_{\omega}(\mathcal{P}^2)}=-\sqrt{32\pi^2}.$$ Furthermore, if $\upsilon: M' \rightarrow M$ is the universal covering of $M$, then there is a 2-torus $T^2$ in $M'$ such that
 \begin{enumerate}
 \item[i)]  $\upsilon(T^2)\subset M^o$,
\item[ii)] $T^2$ represents a non-trivial homological class in $M'$, i.e.  $$0\neq [T^2]\in H^2(M', \mathbb{R}),$$
\item[iii)] the self-intersection number $$T^2\cdot T^2  =0,   \  \  \ and  \ \ \int_{T^2}\varpi =0, $$ for a symplectic form $\varpi$ on $M'$.
\end{enumerate}
\end{theorem}

\begin{proof}
If $H$ is the hyperplane given by $h_0=z_0+z_1+z_2+z_3=0$, then $\varepsilon^i  H$ is defined by
 $h_i= z_0+\varepsilon^i z_1+\varepsilon^{2i}z_2+\varepsilon^{3i}z_3=0$ for $i=0, \cdots, 4$.  Let $X_0=H\cup \varepsilon  H \cup \varepsilon^2  H \cup \varepsilon^3  H \cup \varepsilon^4  H$, which is defined  by $h_0  h_1 h_2    h_3  h_4=0$.

 We claim that the regular locus of $X_0$ consists of 5 copies of  pair-of-pants, i.e. for any $i$, $\varepsilon^iH\backslash \bigcup\limits_{j\neq i}\varepsilon^jH$ is a pair-of-pants $\mathcal{P}^2$.
 Since $\mathbb{Z}_5$ acts on  $X_0$ and switches the irreducible components, we only  need to prove that $H\backslash ( \varepsilon  H \cup \varepsilon^2  H \cup \varepsilon^3  H \cup \varepsilon^4  H)$ is a pair-of-pants. Note that $\varepsilon^2(H\cap  \varepsilon  H \cap \varepsilon^2  H \cap \varepsilon^3  H)=H\cap  \varepsilon^2  H \cap \varepsilon^3  H \cap \varepsilon^4  H$, $\varepsilon^3(H\cap  \varepsilon  H \cap \varepsilon^2  H \cap \varepsilon^3  H)=H\cap  \varepsilon  H \cap \varepsilon^3  H \cap \varepsilon^4  H$, and $\varepsilon^4(H\cap  \varepsilon  H \cap \varepsilon^2  H \cap \varepsilon^3  H)=H\cap  \varepsilon  H \cap \varepsilon^2  H \cap \varepsilon^4  H$.
 Thus we only need to prove that $H\cap  \varepsilon  H \cap \varepsilon^2  H \cap \varepsilon^3  H$ is empty.  The coefficient matrix of the equations $h_0=h_1=h_2=h_3=0$ is the Vandermonde matrix, and the determinant  is equal to $\varepsilon^4(\varepsilon -1)^6(\varepsilon+ 1)^2(\varepsilon^2+\varepsilon+1)\neq 0 $. We obtain the claim.

If we let $$X_t= \{[z_0, z_1, z_2, z_3]\in  \mathbb{CP}^3 | z_0^5+z_1^5+z_2^5+z_3^5+th_0  h_1 h_2    h_3  h_4=0\}, \  \  t\in (1, \infty),$$  then  $X_t$ is invariant under the $\mathbb{Z}_5$-action,  and $X_t$ converges to $X_0$ in $\mathbb{CP}^3$ as varieties  when $t\rightarrow\infty$.  Note that  $[1,0,0,0], [0,1,0,0], [0,0,1,0], [0,0,0,1]$ are  the fixed points when $\mathbb{Z}_5$ acts on $ \mathbb{CP}^3$, and $1+t\varepsilon^j\neq 0$. Thus the $\mathbb{Z}_5$-action on $X_t$ is free.
For $t\gg 1$, there is an open subset $X_t^o \subset X_t$ diffeomorphic to the regular locus of $X_0$, which  consists  of 5 copies of  pair-of-pants, and $X_t^o$ is invariant under the $\mathbb{Z}_5$-action.  Therefore the quotient $X_t^o/\mathbb{Z}_5$ is a   pair-of-pants in $X_t/\mathbb{Z}_5$.

The rest of the assertion  seems a corollary  of Theorem \ref{mainthm}. However we  give a direct calculation proof  since the degeneration  used here  is different from the family of varieties in the proof of Theorem \ref{mainthm}.

We choose new homogeneous  coordinates $Z_0, Z_1, Z_2, Z_3$ on  $\mathbb{CP}^3$ such that    $\varepsilon^i  H $ is defined by $Z_i=0$,  $i=0,1,2,3$.  Equivalently,  $Z_i= z_0+\varepsilon^i z_1+\varepsilon^{2i}z_2+\varepsilon^{3i}z_3$. Then   $\varepsilon^4  H $ is given  by $a_0Z_0+a_1Z_1+a_2Z_2+a_3Z_3=0$,   and  $X_t$ is defined  by $$\bar{f}_t=Z_0Z_1Z_2Z_3(a_0Z_0+a_1Z_1+a_2Z_2+a_3Z_3)+t^{-1}P_5=0, $$ where $a_0Z_0+a_1Z_1+a_2Z_2+a_3Z_3= z_0+\varepsilon^4 z_1+\varepsilon^{8}z_2+\varepsilon^{12}z_3$, and $P_5$ is a homogeneous polynomial of degree 5.    If  $w_j=Z_j/Z_0$, $j=1,2,3$, then
$\varepsilon^j  H \cap \{Z_0\neq 0\}$ is given by $w_j=0$, $j=1,2,3$, and $\varepsilon^4  H \cap \{Z_0\neq 0\}$ is given by $a_0+a_1w_1+a_2w_2+a_3w_3=0$.   We let   $$f_t=\bar{f}_t/Z_0^5= w_1w_2w_3(a_0+a_1w_1+a_2w_2+a_3w_3)+t^{-1}P_5',$$  where $P_5'=P_5/Z_0^5$.  Note that $a_0\neq 0$ since, otherwise, $(0,0,0)$ solves $w_i=0$, $i=1,2,3$, and $a_0+a_1w_1+a_2w_2+a_3w_3=0$, which is a contradiction.

  We consider the hyperplane
$\varepsilon^3  H \cap \{Z_0\neq 0\}\subset \mathbb{C}^3$, and  $\varepsilon^3  H \cap \{Z_0\neq 0\} \backslash (\varepsilon^1  H\cup \varepsilon^2  H \cup \varepsilon^4  H)$ is the pair-of-pants $\mathcal{P}^2$ belonging to the irreducible component $\varepsilon^3  H$.  We choose a 2-torus $$T^2=\{(w_1, w_2)\in \mathcal{P}^2| |w_1|=r_1, |w_2|=r_2 \}\subset \varepsilon^3  H,$$ for certain $r_1, r_2 \in\mathbb{R}^1\backslash \{0\}$. If $\theta_i$, $i=1,2$, is the angle of $w_i$, i.e. $w_i=r_i\exp \sqrt{-1}\theta_i$, then $\theta_1$ and $\theta_2$ are angular coordinates on $T^2$.
  If $V$ is a small neighbourhood of  $T^2$ in $\mathbb{CP}^3$, then $X_t \cap V$ is given by $f_t=0$ and $\varepsilon^3  H \cap V$ is given by $w_3=0$. We choose $V$ such that $V$ does not intersect with $\varepsilon  H $, $\varepsilon^2  H $, and $\varepsilon^4  H $.

Since the meromorphic form $$\Omega = \frac{Z_0^2Z_1Z_2Z_3dw_1\wedge dw_2 \wedge dw_3}{\bar{f}_t w_1w_2 w_3}$$ has only a  simple pole along $X_t$,  the Poincar\'{e} residue formula gives a holomorphic 2-form $\Omega_t={\rm res}_{X_t}(\Omega)$ on $X_t$, which defines a non-trivial cohomological  class, i.e. $$0\neq [\Omega_t]\in H^2(X_t, \mathbb{C}).$$
 Note that   $$\Omega_t=
\frac{ dw_1\wedge dw_2}{\frac{\partial f_t}{\partial w_3}} =
 \frac{w_1w_2d\log(w_1)\wedge d\log(w_2)}{\frac{\partial f_t}{\partial w_3}},$$  on $X_t \cap V$.     We calculate
  \begin{align*}\frac{\partial f_t}{\partial w_3}& =w_1w_2(a_0+a_1w_1+a_2w_2)+2a_3w_1w_2w_3+t^{-1}\frac{\partial P_5'}{\partial w_3},\end{align*}  and thus $$\Omega_t=  \frac{d\log(w_1)\wedge d\log(w_2)}{a_0+a_1w_1+a_2w_2+2a_3w_3+t^{-1}w_1^{-1}w_2^{-1}\frac{\partial P_5'}{\partial w_3}}. $$

  We choose  $ r_1$ and $r_2$ such that $|a_1|r_1+|a_2|r_2 \ll \frac{1}{2}|a_0|$ since $a_0\neq 0$, and $r_1r_2> \epsilon >0$ for a small constant $\epsilon>0$. Since $|a_0+a_1w_1+a_2w_2+a_3w_3| >\epsilon'>0$ for any $w\in V$ and a constant $\epsilon'>0$, the equation  $f_t=0$ shows  that
  $$\epsilon\epsilon' |w_3|_{X_t \cap V}| < |P_5' |t^{-1}\rightarrow 0,$$  when $t \rightarrow 0$.
    For any isotopic embedding $\phi_t: \mathcal{P}^2 \hookrightarrow X_t$ with $\phi_\infty={\rm id}$,  \begin{align*}\int_{\phi_t(T^2)}\Omega_t \rightarrow & \int_{T^2}\frac{d\theta_1 \wedge d\theta_2}{a_0 + a_1w_1 +a_2w_2} \\ = & \int_{T^2}\frac{(\overline{a_0} + \overline{a_1}r_1e^{-\sqrt{-1}\theta_1} +\overline{a_2}r_2e^{-\sqrt{-1}\theta_2})d\theta_1 \wedge d\theta_2}{|a_0 + a_1r_1e^{\sqrt{-1}\theta_1} +a_2r_2e^{\sqrt{-1}\theta_2}|^2}
    \neq 0,\end{align*} when $t\rightarrow\infty$.  Thus for $t\gg 1$,   $\phi_t(T^2) \subset \phi_t(\mathcal{P}^2) \subset X_t^o $, and  $$0\neq [ \phi_t(T^2)] \in H_2(X_t, \mathbb{R}).$$

By varying $r_1$ and $r_2$, we obtain that  the self-intersection number $\phi_t(T^2) \cdot \phi_t(T^2)=0 $. If $\varpi$ is a toric symplectic form on $\mathbb{CP}^3$, then $\varpi$ is preserved  by  the $\mathbb{Z}_5$-action, and $\varpi|_{T^2}\equiv 0$. Therefore $$\int_{ \phi_t(T^2)} \varpi =\int_{T^2} \varpi=0.$$
Note that
if $\upsilon: X_t \rightarrow X_t/\mathbb{Z}_5$ denotes the quotient map, then  $\upsilon: \phi_t(\mathcal{P}^2) \rightarrow M^o$ is a diffeomorphism.
   We obtain the conclusion since $M$ is diffeomorphic to $ X_t/\mathbb{Z}_5$,  and $M'=X_t$.
\end{proof}

The difference between  the assertion ii) in this theorem  and iv) in Theorem \ref{mainthm} is that the torus $\upsilon(T^2)$  is homological zero since $\Omega_t$ is not preserved by the $\mathbb{Z}_5$-action, and the geometric genus $p_g(\tilde{X})=0$. However $\upsilon^{-1}(\upsilon(T^2))$ consists of 5 copies of
 $T^2$  in the universal covering, and each connected component  represents a non-trivial class in $M'$.   The assertion ii) is also different from the case of $\mathcal{P}^2 \subset \mathbb{CP}^2$, where $\mathbb{CP}^2$ is simply  connected, and  numerical Lagrangian 2-tori  in   $\mathcal{P}^2$ are homological zero. Moreover, Theorem \ref{thm+0000} can be applied to the current case, and shows that the assertions  ii) and iii) of Theorem  \ref{pro2}   imply  directly the Yamabe invariant  $\sigma(M)\leq 0$.    Therefore we would  like to think that  Theorem  \ref{pro2} provides a sensible decomposition of the Godeaux surface.

\section{Pair-of-pants vs K3 surfaces}
The goal of this section is to prove the assertions i), ii), iii), and v) of  Theorem \ref{mainthm}, which depends heavily on Mikhalkin's paper \cite{Mik} and tropical geometry (cf.  Chapter 1 of \cite{Gro} and \cite{Mik2}). We review  the relevant facts first.

\subsection{Amoebas}
$(\mathbb{C}^*)^{3}$ denotes  the complex torus of dimension 3, and for any $m =(m_1, m_2, m_3)\in\mathbb{Z}^{3}$, $w^m=w_1^{m_1}w_2^{m_2} w_{3}^{m_{3}}$ where $w_1, w_2, w_{3}$ are coordinates on $(\mathbb{C}^*)^{3}$.
Let $X^o$ be the  hypersurface in $(\mathbb{C}^*)^{3}$ defined by a  Laurent  polynomial  \begin{equation}\label{eq4.1}f(z)=\sum_{m\in S}a_m w^m =0,\end{equation} where  $a_m \in \mathbb{C}$, and $S\subset \mathbb{Z}^{3}$  is finite. The Newton polytope $\Delta \subset\mathbb{R}^{3}$ of $X^o$ is defined as the convex hull of $m \in S$ such that $a_m \neq 0$.
We consider the log-map $${\rm Log}: (\mathbb{C}^*)^{3} \rightarrow \mathbb{R}^{3},  \  \   {\rm Log}(w)=(\log |w_1|, \log |w_2|, \log |w_{3}|).$$ The amoeba of $X^o$ is defined in  \cite{GKZ} (see also \cite{Mik2,PR} for analytic treatments) as the image  $$\mathcal{A}={\rm Log}(X^o)\subset \mathbb{R}^{3}. $$

Let $v: \Delta(\mathbb{Z})\rightarrow \mathbb{R}$ be a function where $\Delta(\mathbb{Z})=\Delta\cap \mathbb{Z}^{3}$ is the set of lattice points of $\Delta$. $v$ induces a  rational  polyhedral subdivision $\mathcal{T}_v$ of $\Delta$  as follows.
If $\tilde{\Delta}$ is the upper convex hull of $\{(m, v(m))| m \in \Delta(\mathbb{Z})\}$ in $\Delta\times\mathbb{R}$, i.e. $$\tilde{\Delta}=\{(a, b)\in \Delta\times\mathbb{R}| \exists b'\in {\rm Cov}\{(m, v(m))| m \in \Delta(\mathbb{Z})\}, \ \ b\geq b'\},$$
 then $\mathcal{T}_v$ is the set of the images of proper faces  under  the projection   $\tilde{\Delta}\rightarrow \Delta$ (cf. Chapter 1 in \cite{Gro}).
  The discrete  Legendre transform $L_v: \mathbb{R}^{3} \rightarrow \mathbb{R}$ of $v$ is defined
 as \begin{equation}\label{eq4.2}L_v(x)=\max_{m\in \Delta(\mathbb{Z})}\{l_m(x)\},  \ \ \   \ l_m(x)=\langle m, x\rangle -v(m), \end{equation} where  $\langle \cdot, \cdot\rangle$ is the Euclidean metric. Then  $L_v$ is a convex  piecewise linear  function.

  The non-smooth locus $\Pi_v$ of $L_v$ is called the tropical hypersurface defined by   $v$, or the  non-archimedean amoeba, which  is a balanced polyhedral complex dual to the  subdivision $\mathcal{T}_v$ of $\Delta$ (cf. Proposition 2.1 in \cite{Mik}). $\Pi_v$ is a strong deformational retract of the    amoeba $\mathcal{A}$ (cf.   \cite{Mik2,PR}).   Furthermore, there is a one-to-one corresponding between the set of  $k$-dimensional cells $\check{\rho} \subset \Pi_v$, $k=0, 1, 2$,  and  the set of  $(3-k)$-dimensional cell $\rho$ in $\mathcal{T}_v$, i.e. \begin{equation}\label{eq4.3}\{k-{\rm cells} \ \check{\rho} \subset \Pi_v \} = \{(3-k)-{\rm cells} \ \rho\in\mathcal{T}_v\}, \ \ {\rm via}\ \ \check{\rho}\mapsto \rho, \end{equation}  and the cell $\check{\rho} $ is unbounded if and only if $\rho\subset \partial \Delta$.  The vertices of $\mathcal{T}_v $ are one-to-one corresponding to the connected components of the complement of  $\Pi_v$. Two cells $\rho_1 \subset \rho_2\in \mathcal{T}_v$ if and only if $\check{\rho}_1 \supset  \check{\rho}_2$ in $\Pi_v$.

We deform $X^o$ via the Viro's patchworking polynomials  (cf. \cite{Viro}). The patchworking family $X^o_t$, $t\in (1, \infty)$, is defined  by $$f_t(w)=\sum_{m\in S}a_m t^{-v(m)}w^m =0.$$ If the deformed log map ${\rm Log}_t: (\mathbb{C}^*)^{3} \rightarrow \mathbb{R}^{3}$ is defined by \begin{equation}\label{eq-log}{\rm Log}_t(w)=\Big(\frac{\log |w_1|}{\log t}, \frac{\log |w_2|}{\log t}, \frac{\log |w_{3}|}{\log t}\Big),\end{equation} then we have a family of  the amoebas $$\mathcal{A}_t={\rm Log}_t(X^o_t)\subset \mathbb{R}^{3}.$$  Corollary 6.4 in \cite{Mik} asserts   that $$ \mathcal{A}_t \rightarrow \Pi_v,  \  \  \  {\rm when} \ \ t\rightarrow\infty, $$ in the Hausdorff sense. Therefore, $\Pi_v$ is a good approximation of $\mathcal{A}_t$ when $t\gg 1$.

If
 \begin{align*}H^o & = \{(w_1, w_2, w_3)\in (\mathbb{C}^*)^3|1+ w_1+w_2+w_3=0\} \\ & =\{( w_1, w_2)\in (\mathbb{C}^*)^2|1+ w_1+w_2\neq 0\}, \end{align*}
 then $H^o$ is the 4-dimensional pair-of-pants, i.e.  $H^o=\mathcal{P}^2$. The Newton polytope $\Delta$ is the standard simplex in $\mathbb{R}^3$, i.e. $$\Delta=\{(x_1,x_2,x_3)\in \mathbb{R}^3|x_1+x_2+x_3\leq 1, x_i\geq 0, i=1,2,3\}.$$  If $v\equiv 0$, then the corresponding  tropical hypersurface $\Pi_v$ is the non-linear locus of $L_v(x)=\max \{0, x_1,x_2,x_3\}$.

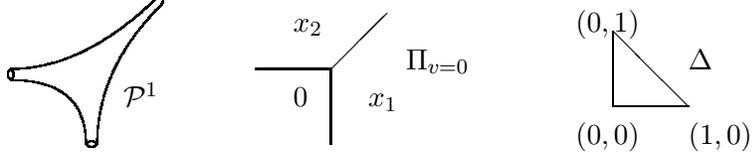
\begin{figure}\label{pic3}
{\footnotesize
\caption{Tropical 2-dimensional pair-of-pants $\mathcal{P}^1$.} }
 \begin{center}
 \setlength{\unitlength}{0.5cm}
 \begin{picture}(8,4)(-2,-2)
\qbezier(0,0)(2,0)(4,2)
\qbezier(0,-0.3)(2,-0.3)(2,-2)
\qbezier(2.3,-2)(2.3,0)(4,1.7)
\qbezier(4,1.7)(4.1,1.8)(4,2)
\qbezier(0,-0.3)(0.15,-0.15)(0,0)
\qbezier(0,-0.3)(-0.15,-0.15)(0,0)
\qbezier(2,-2)(2.15,-1.85)(2.3,-2)
\qbezier(2,-2)(2.15,-2.15)(2.3,-2)
\put(3,-1){$\mathcal{P}^1$}
\end{picture}
  \setlength{\unitlength}{0.5cm}
  \begin{picture}(8,4)(-2,-2)
\put(0,0){\line(1,1){1.5}}
\put(0,0){\line(0,-1){2}}
\put(0,0){\line(-1, 0){2}}
\put(-1,-1){$0$}
\put(-1,1){$x_2$}
\put(1,-1){$x_1$}
\put(2,0){$\Pi_{v=0}$}
\end{picture}
\setlength{\unitlength}{0.5cm}
  \begin{picture}(8,4)(-2,-2)
\put(-1,-1){\line(1,0){2}}
\put(-1,-1){\line(0, 1){2}}
\put(1,-1){\line(-1,1){2}}
\put(-2,-2){$(0,0)$}
\put(1,-2){$(1,0)$}
\put(-2,1){$(0,1)$}
\put(1,0){$\Delta$}
\end{picture}
 \end{center}
 \end{figure}

 \begin{figure}\label{pic+}
{\footnotesize
\caption{$\Pi_v$ vs $\mathcal{T}_v$. } }
 \begin{center}
  \setlength{\unitlength}{0.5cm}
  \begin{picture}(10,8)(-1,-1)
  \put(2,-1){$\Pi_v$}
\put(0,0){\line(1,1){1}}
\put(0,0){\line(-1,0){1}}
\put(0,0){\line(0,-1){1}}
\put(1,1){\line(1,0){1}}
\put(1,1){\line(0,1){1}}
\put(2,1){\line(1,1){1}}
\put(2,1){\line(0,-1){1}}
\put(1,2){\line(-1,0){1}}
\put(1,2){\line(1,1){1}}
\put(3,2){\line(1,0){1}}
\put(3,2){\line(0,1){1}}
\put(2,3){\line(1,0){1}}
\put(2,3){\line(0,1){1}}
\put(3,2){\line(1,0){1}}
\put(3,2){\circle*{0.5}}
\put(2,3){\circle*{0.5}}
\put(3,3){\circle*{0.5}}
\put(4,4){\circle*{0.5}}
\put(3,3){\line(1,1){1}}
\put(4,4){\line(0,1){1}}
\put(4,4){\line(1,0){1}}
\put(4,2){\line(0,-1){1}}
\put(4,2){\line(1,1){1}}
\put(5,3){\line(1,0){1}}
\put(5,3){\line(0,1){1}}
\put(5,4){\line(1,1){1}}
\put(6,3){\line(1,1){1}}
\put(6,3){\line(0,-1){1}}
\put(2,4){\line(-1,0){1}}
\put(2,4){\line(1,1){1}}
\put(3,5){\line(0,1){1}}
\put(3,5){\line(1,0){1}}
\put(4,5){\line(1,1){1}}
\put(3,6){\line(1,1){1}}
\put(3,6){\line(-1,0){1}}
\end{picture}
 \setlength{\unitlength}{0.5cm}
  \begin{picture}(10,8)(0,0)
\put(0,0){\line(1,0){8}}
\put(0,0){\line(0,1){8}}
\put(8,0){\line(-1,1){8}}
\put(0,2){\line(1,0){6}}
\put(2,0){\line(0,1){6}}
\put(6,0){\line(-1,1){6}}
\put(4,0){\line(-1,1){4}}
\put(4,0){\line(0,1){4}}
\put(6,0){\line(0,1){2}}
\put(0,4){\line(1,0){4}}
\put(0,6){\line(1,0){2}}
\put(2,0){\line(-1,1){2}}
\put(5,5){$\Delta_4$}
\put(4,6){$\mathcal{T}_v$}
\put(2.5,2.5){\circle*{0.5}}
\put(3.5,3.5){\circle*{0.5}}
\put(1.5,3.5){\circle*{0.5}}
\put(3.5,1.5){\circle*{0.5}}
\end{picture}
 \end{center}
 \end{figure}
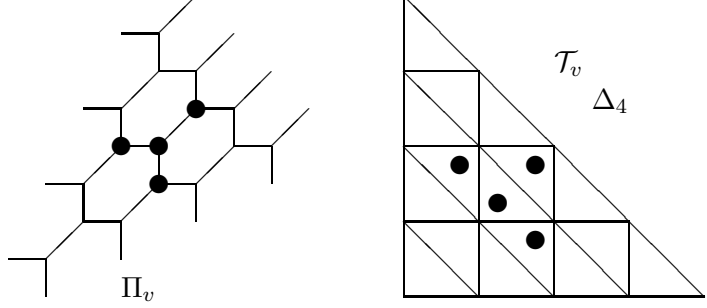

\subsection{Decompositions}
Let $M$ be a smooth manifold diffeomorphic to a hypersurface $X$ in $\mathbb{CP}^{3}$ of degree $d\geq 5$.   We regard $\mathbb{CP}^3$ as a toric manifold by $(\mathbb{C}^*)^3 \subset \mathbb{CP}^3$. For any positive integer $d$,  the corresponding  polytope $\Delta_d$ of $(\mathbb{CP}^3, \mathcal{O}_{\mathbb{CP}^3}(d))$  is the standard simplex in $\mathbb{R}^3$, i.e. \begin{equation}\label{eq4.2.1}\Delta_d=\{(x_1, x_2, x_3)\in \mathbb{R}^3|x_1+x_2+ x_3\leq d, x_i\geq 0, i=1,2,3\}.\end{equation}     If  $\Delta_d(\mathbb{Z})=\Delta_d \cap \mathbb{Z}^3$,  then for any $m \in \Delta_d(\mathbb{Z})$, $w^m=w_1^{m_1}w_2^{m_2}w_3^{m_3}$ is viewed as  a section of $\mathcal{O}_{\mathbb{CP}^3}(d)$.

 Let $\Delta_d^o$ be  the convex hull of ${\rm int}(\Delta_d) \cap \mathbb{Z}^3$, i.e. $\Delta_d^o={\rm Cov}({\rm int}\Delta_d \cap \mathbb{Z}^3)$, and $\Delta_d^o(\mathbb{Z})=\Delta_d^o \cap  \mathbb{Z}^3$. Note that $\Delta_d^o$ is a translation of the simplex $\Delta_{d-4}$, i.e. \begin{equation}\label{eq4.2.2}\Delta_d^o=(1,1,1)+\Delta_{d-4}.\end{equation}
If $p_g(d)$ denotes the geometric genus of $X$ in $\mathbb{CP}^3$, then
$$p_g(d)=\frac{(d-1)(d-2)(d-3)}{6}=\sharp \Delta_d^o(\mathbb{Z}),$$ by the Noether's formula $\chi(X)=12(1-p_g(d))- K_X^2$, since $X$ is simply connected,   $K_X^2=d(d-4)^2$ and the Euler number  $\chi(X)=d^3-4d^2+6d$.

Let $v: \Delta_d(\mathbb{Z}) \rightarrow \mathbb{R}$ be a function such that the induced subdivision  $\mathcal{T}_v$ is a uni-modular lattice triangulation, i.e. any 3-cell in $\mathcal{T}_v$ is a simplex with Euclidean  volume $\frac{1}{6}$. Since the  Euclidean  volume of $\Delta_d$ is $\frac{d^3}{6}$, $\mathcal{T}_v$ has $d^3$ 3-dimensional cells, and  $ \Pi_v$ has $d^3$   vertices.
$v$ defines a   patchworking family  $X_t$, $t\in (1, \infty)$,  given by \begin{equation}\label{eq4.-001}f_t(w)=\sum\limits_{m\in \Delta_d(\mathbb{Z})}t^{-v(m)}w^m=0.\end{equation}
Since any smooth hypersurface of degree $d$ is diffeomorphic to the same differential manifold, a smooth $X_t$ is diffeomorphic to $M$. We let  $X_t^o=X_t\cap (\mathbb{C}^*)^3$.

  In \cite{Mik},  Mikhalkin proved that $M$ admits a pair-of-pants decomposition.

\begin{theorem}[\cite{Mik}]\label{M}
  There is a fibration $\lambda: X_t^o \rightarrow \Pi_v$ satisfying:
   \begin{enumerate}
\item[i)] For any vertex $\check{q}\in \Pi_v$, let  $U_{\check{q}} \subset \Pi_v$ be the interior of the star of the  barycentre triangulation  of $\Pi_v$, called a primitive piece associated with $\check{q}$. Then  $U_{\check{q}} \cap U_{\check{q}'}$ is empty if $\check{q}\neq \check{q}'$, and $\lambda^{-1}(U_{\check{q}})$ is diffeomorphic to the  pair-of-pants $\mathcal{P}^2$.
\item[ii)]  $$\coprod\limits_{\check{q} \in \{{\rm vertices \  in} \ \Pi_v\}}\lambda^{-1}(U_{\check{q}})\subset X_t^o$$  is open dense. Therefore $X_t$ admits a pair-of-pants decomposition consisting of $d^3$ copies of $\mathcal{P}^2$.
\item[iii)]  If a point  $x$ belongs to the interior of  a 2-cell in $\Pi_v$, then the fibre  $f^{-1}(x)$ is a torus $T^2$, and  if  $x$ is in the  interior of a 1-cell of $\Pi_v$, then  $f^{-1}(x)$ has the  shape of $\ominus \times \bigcirc$.
\end{enumerate}
\end{theorem}

The goal of Theorem \ref{mainthm} is to reduce the number of pair-of-pants from $d^3$ to $d(d-4)^2$ by showing that those extra pair-of-pants form certain subsets of  K3 surfaces. To decompose $X_t$ into  a union of subsets of Calabi-Yau manifolds was previously  studied by Leung and Wan in \cite{LW}.

\begin{figure}\label{pic4}
{\footnotesize
\caption{A decomposition of projective curve of degree 4 in $\mathbb{CP}^2$, which consists of 4 copies of pair-of-pants $\mathcal{P}^1$, and 3 copies of $(0,1)\times S^1 \subset \mathbb{C}^*$.} }
 \begin{center}
  \setlength{\unitlength}{0.5cm}
  \begin{picture}(10,8)(-1,-1)
\put(0,0){\line(1,1){1}}
\put(0,0){\line(-1,0){1}}
\put(0,0){\line(0,-1){1}}
\put(1,1){\line(1,0){1}}
\put(1,1){\line(0,1){1}}
\put(2,1){\line(1,1){1}}
\put(2,1){\line(0,-1){1}}
\put(1,2){\line(-1,0){1}}
\put(1,2){\line(1,1){1}}
\put(3,2){\line(1,0){1}}
\put(3,2){\line(0,1){1}}
\put(2,3){\line(1,0){1}}
\put(2,3){\line(0,1){1}}
\put(3,2){\line(1,0){1}}
\put(3,2){\circle*{0.5}}
\put(2,3){\circle*{0.5}}
\put(3,3){\circle*{0.5}}
\put(4,4){\circle*{0.5}}
\put(3,3){\line(1,1){1}}
\put(4,4){\line(0,1){1}}
\put(4,4){\line(1,0){1}}
\put(4,2){\line(0,-1){1}}
\put(4,2){\line(1,1){1}}
\put(5,3){\line(1,0){1}}
\put(5,3){\line(0,1){1}}
\put(5,4){\line(1,1){1}}
\put(6,3){\line(1,1){1}}
\put(6,3){\line(0,-1){1}}
\put(2,4){\line(-1,0){1}}
\put(2,4){\line(1,1){1}}
\put(3,5){\line(0,1){1}}
\put(3,5){\line(1,0){1}}
\put(4,5){\line(1,1){1}}
\put(3,6){\line(1,1){1}}
\put(3,6){\line(-1,0){1}}
\end{picture}
 \setlength{\unitlength}{0.5cm}
  \begin{picture}(10,9)(-1,-1)
\put(0,0){\line(1,1){1}}
\put(0,0){\line(-1,0){1}}
\put(0,0){\line(0,-1){1}}
\put(1,1){\line(1,0){1}}
\put(1,1){\line(0,1){1}}
\put(2,1){\line(1,1){1}}
\put(2,1){\line(0,-1){1}}
\put(1,2){\line(-1,0){1}}
\put(1,2){\line(1,1){1}}

\put(4,2){\line(1,0){1}}
\put(4,2){\line(0,1){1}}
\put(3,3){\line(1,0){1}}
\put(3,3){\line(0,1){1}}
\put(4,2){\line(1,0){1}}
\put(4,2){\circle*{0.5}}
\put(3,3){\circle*{0.5}}
\put(4,3){\circle*{0.5}}
\put(5,4){\circle*{0.5}}
\put(4,3){\line(1,1){1}}
\put(5,4){\line(0,1){1}}
\put(5,4){\line(1,0){1}}
\put(4,2){\line(-1,-1){1}}
\put(3,3){\line(-1,-1){1}}

\put(6,1){\line(-1,0){1}}
\put(6,1){\line(0,-1){1}}
\put(6,1){\line(1,1){1}}
\put(7,2){\line(1,0){1}}
\put(7,2){\line(0,1){1}}
\put(7,3){\line(1,1){1}}
\put(8,2){\line(1,1){1}}
\put(8,2){\line(0,-1){1}}
\put(7,3){\line(-1,0){1}}

\put(2,5){\line(-1,0){1}}
\put(2,5){\line(1,1){1}}
\put(3,6){\line(0,1){1}}
\put(3,6){\line(1,0){1}}
\put(4,6){\line(1,1){1}}
\put(3,7){\line(1,1){1}}
\put(3,7){\line(-1,0){1}}
\put(2,5){\line(0,-1){1}}
\put(4,6){\line(0,-1){1}}
\end{picture}
 \end{center}
 \end{figure}
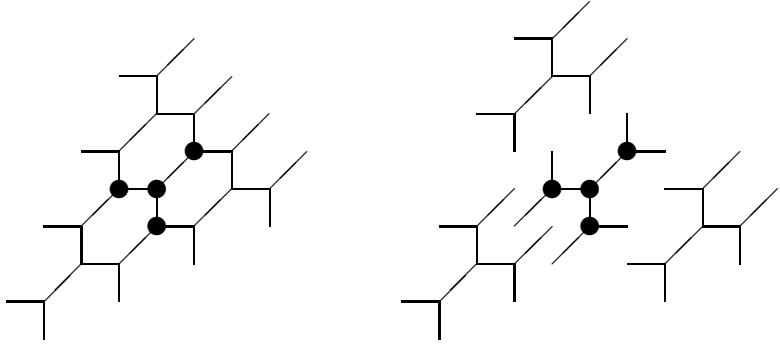

We remark that there is an alternative way to view the pair-of-pants decomposition in the current case. Assume that any vertex $\check{q}\in\Pi_v$ is an integral  vector, i.e. $\check{q} \in\mathbb{Z}^3$, and for any $m\in \Delta_d(\mathbb{Z})$, $v(m)\in\mathbb{Z}$.
   The pair $(\Delta_d, v)$ induces a toric degeneration $\mathcal{X}\rightarrow \mathbb{C}$ with a relative ample line bundle $ \mathcal{O}_{\mathcal{X}/\mathbb{C}}(d)$ (cf. Chapter 1 of \cite{Gro} and \cite{GS1}) satisfying:    \begin{enumerate}
\item[i)] Any generic fibre $\mathcal{X}_{w_0}$, $w_0\neq 0$, is $\mathbb{CP}^3$, and the restriction of  $\mathcal{O}_{\mathcal{X}/\mathbb{C}}(d)$  on $ \mathcal{X}_{w_0}$ is $\mathcal{O}_{\mathbb{CP}^3}(d)$.
\item[ii)] The central fibre $\mathcal{X}_0$ is a singular variety  consisting of $d^3$ copies  of  $\mathbb{CP}^3$ as  irreducible components.  The singular set of $\mathcal{X}_0$ belongs to  the union of toric boundary divisors  of $\mathbb{CP}^3$.
   \item[iii)]  There is a one-to-one corresponding between 3-cells $\rho$ of $\mathcal{T}_v$ and irreducible components $\mathcal{X}_{0, \rho}$ of $\mathcal{X}_0$. Furthermore, the pair $(\Delta_d, \mathcal{T}_v)$ can be regarded as the intersection complex of $ \mathcal{X}_0$.
 \item[iv)]   For any $m\in \Delta_d(\mathbb{Z})$, $Z_m=w_0^{v(m)}w^m$ defines a section of $\mathcal{O}_{\mathcal{X}/\mathbb{C}}(d)$ such that the restriction of $Z_m$ on the irreducible component $\mathcal{X}_{0, \rho}$ is  identical to zero if $m$ does not belong to $ \rho$. Furthermore, if $\{m^0,m^1,m^2,m^3\}=\rho \cap \Delta_d(\mathbb{Z})$, then $Z_{m^i}$, $i=0,1,2,3$, are homogeneous coordinates on $ \mathbb{CP}^3$.  $Z_m$ are generalised theta functions in the Gross-Siebert sense (cf. \cite{GS2}).
 \item[v)] There is a  divisor $\mathcal{D}\subset \mathcal{X}$ such that the intersection of $\mathcal{D}$ with any generic fibre $\mathcal{X}_{w_0}$ is the toric boundary divisor of $\mathcal{X}_{w_0}$. Moreover, if $\mathcal{X}_0^o$ denotes the regular locus  of  $\mathcal{X}_0$, then
      $\mathcal{X}_0^o\backslash \mathcal{D}$ is the disjoint union of $d^3$ copies of $(\mathbb{C}^*)^3$.
\end{enumerate}
The patchworking polynomial (\ref{eq4.-001}) reads $$ \sum_{m\in \Delta_d(\mathbb{Z})} Z_m=0, $$ which defines a subvariety of $\mathcal{X}$ over $ \mathbb{C}$, denoted   as $\tilde{\mathcal{X}} \rightarrow  \mathbb{C} $. When $w_0=t^{-1}$, the fibre $\tilde{\mathcal{X}}_{w_0}=X_t$, and therefore $\tilde{\mathcal{X}}$ can be regarded as an extension of the patchworking family. For any irreducible component $\mathcal{X}_{0, \rho}$ in the central fibre, $\tilde{\mathcal{X}}_{0}\cap \mathcal{X}_{0, \rho}$ is given by $$Z_{m^0}+Z_{m^1}+Z_{m^2}+Z_{m^3}=0,$$ where $\{m^0,m^1,m^2,m^3\}=\rho \cap \Delta_d(\mathbb{Z})$, and  $\tilde{\mathcal{X}}_{0}\cap \mathcal{X}_{0, \rho}\cap (\mathcal{X}_0^o\backslash \mathcal{D})$ is the pair-of-pants $\mathcal{P}^2$.   Thus $\tilde{\mathcal{X}}_{0}\cap (\mathcal{X}_0^o\backslash \mathcal{D}) $ consists of $d^3$ copies of pair-of-pants $\mathcal{P}^2$, and can  be diffeomorphically  embedded into $X_t$ for $t\gg 1$.

\subsection{Proofs of i), ii), iii), and v) in   Theorem \ref{mainthm}}
We fix   a function $v: \mathbb{Z}^3 \rightarrow \mathbb{R}$ by    \begin{equation}\label{eq4.3.1}v(m)= 4\sum_{i=1}^{3}m_i^2 +(2m_1+2m_2+3 m_3)^2, \end{equation} for any $m=(m_1, m_2, m_3)\in \mathbb{Z}^3$, which is the restriction of a positive defined quadratic form $\tilde{v} $ on $ \mathbb{R}^3=\mathbb{Z}^3\otimes_{\mathbb{Z}}\mathbb{R}$.   $v$ induces a Delaunay polyhedral decomposition $\mathrm{Del}_v $ of  $\mathbb{R}^3$ (cf. \cite{AN}). Since $v$ has a $\mathbb{Z}_2$-symmetry by switching $x_1$ and $x_2$,  $\mathrm{Del}_v $ is invariant under the $\mathbb{Z}_2$-action.
  By Lemma 1.8 of  \cite{AN},  $\mathrm{Del}_v $ is obtained by projecting the facets of the convex hull of countable points in the paraboloid $\{(m,v(m))\in \mathbb{Z}^3 \times \mathbb{R}| m \in  \mathbb{Z}^3\} \subset \{(x,\tilde{v}(x))\in \mathbb{R}^3 \times \mathbb{R}| x \in  \mathbb{R}^3\}$. If $\rho \in \mathrm{Del}_v $ is a 3-cell, then there is a linear affine function $\ell_\rho: \mathbb{R}^3 \rightarrow \mathbb{R}$, called the supporting function of $\rho$, such that $\ell_\rho(m)=v(m)$ for $m \in \rho\cap \mathbb{Z}^3$, and $\ell_\rho(m)<v(m)$ for $m \in (\mathbb{R}^3\backslash \rho)\cap \mathbb{Z}^3$.

 It is well-known that  $\mathrm{Del}_v $ is  invariant under the translation by  $m\in \mathbb{Z}^3$ as follows.  Since, for any $m=(m_1, m_2, m_3) \in \mathbb{Z}^3$,  $$ v(m'+m)= v(m')+\sum_{i=1}^{3}(8m_i+2(2m_1+2m_2+3 m_3)) m_i'+v(m),$$  $ v(m'+m)- v(m')$ is a linear affine function of $m'$, and $v(m'+m) $ induces the same decomposition $\mathrm{Del}_v $, i.e. $\{m+ \rho |$ cells $\rho\in \mathrm{Del}_v\} = \mathrm{Del}_v$.

  \begin{lemma}\label{le01}  For any integer $\tilde{d}\geq 1$, the restriction of $\mathrm{Del}_v $ on $\Delta_{\tilde{d}}$ is a unimodular  subdivision $\mathcal{T}_v$  of $\Delta_{\tilde{d}}$, which has    $\tilde{d}^3$ simplices of dimension 3.
\end{lemma}

\begin{proof}
Table 1 shows the restriction of the  decomposition $\mathrm{Del}_v $ on the unit cube $[0, 1]^3$, which consists of 6 simplices of dimension 3.    We let  $p_0=(0,0,0)$, $p_1=(1,0,0)$, $p_2=(0,1,0)$, $p_3=(0,0,1)$, $p_{12}=(1,1,0)$, $p_{13}=(1,0,1)$, $p_{23}=(0,1,1)$ and $p_{123}=(1,1,1)$, and    $\overline{p_ip_jp_kp_l}$ be  the convex hull of points $p_i, p_j, p_k, p_l$, i.e.   $\overline{p_ip_jp_kp_l}={\rm Cov}(p_i, p_{j}, p_{k}, p_{l})$.

\begin{table}[h!]
{\footnotesize
\caption{Decomposition $\mathrm{Del}_v $ on $[0,1]^3$. }
 \begin{tabular}{  p{2cm}| p{3.5cm} | l | l | l| l| l| l| l| l}
\hline
3-cells  & functions & $p_0 $ & $p_1 $ & $p_2$ & $p_3$ & $p_{12}$ & $p_{13}$ & $p_{23}$ & $p_{123}$ \\ \hline
&  $v$ & 0 & 8 & 8 & 13 & 24 & 33 & 33 & 61\\
 \hline
 $\overline{p_0p_{1}p_{2}p_{3}}$ &  $8x_1+8x_2+13x_3$  & 0 & 8 & 8 & 13 & 16 & 21 & 21 & 29\\
 \hline
 $\overline{p_{12} p_{1}p_{2} p_{3}}$ &  $16x_1+16x_2+21x_3-8$  & --8 & 8 & 8 & 13 & 24 & 29 & 29 & 45\\
 \hline
  $\overline{p_{2} p_{23} p_{12} p_{3}}$ &  $16x_1+20x_2+25x_3-12$  & --12 & 4 & 8 & 13 & 24 & 29 & 33 & 49\\
 \hline
 $\overline{p_{1}p_{13} p_{12} p_{3}}$ &  $20x_1+16x_2+25x_3-12$  & --12 & 8 & 4 & 13 & 24 & 33 & 29 & 49\\
 \hline
 $\overline{p_{3}p_{13}p_{23}p_{12}}$ &  $20x_1+20x_2+29x_3-16$  & --16 & 4 & 4 & 13 & 24 & 33 & 33 & 53\\
 \hline
 $\overline{p_{12}p_{13}p_{23} p_{123}}$ &  $28x_1+28x_2+37x_3-32$  & --32 & --4 & --4 & 5 & 24 & 33 & 33 & 61\\
 \hline
 \end{tabular}
}
\end{table}

We need to verify that the supporting functions do not support any lattice points outside the cube. Because of the $\mathbb{Z}^3$-translation symmetry and $\mathbb{Z}_2$-reflection symmetry of $\mathrm{Del}_v $, we check the nearby points
 $p_1'=(0,-1,0)$, $p_2'=(1,-1,0)$, $p_3'=(1,-1,1)$, $p_4'=(0,-1,1)$, $p_5'=(0,0,-1)$, $p_6'=(1,0,-1)$, $p_7'=(1,1,-1)$ and $p_8'=(0,1,-1)$. Table 2 presents the data.

\begin{table}[h!]
{\footnotesize
\caption{Values  of the functions on  nearby lattice points.}
 \begin{tabular}{  p{2cm}| p{3.5cm} | l | l | l| l| l| l| l| l}
\hline
3-cells  & functions & $p_1' $ & $p_2' $ & $p_3'$ & $p_4'$ & $p_5'$ & $p_6'$ & $p_7'$ & $p_8'$ \\ \hline
&  $v$ & 8 & 8 & 21 & 9 & 13 & 9 & 13 & 9\\
 \hline
 $\overline{p_0p_{1}p_{2}p_{3}}$ &  $8x_1+8x_2+13x_3$  & --8 & 0 & 13 & 5 & --13 & --5 & 3 & --5\\
 \hline
 $\overline{p_{12} p_{1}p_{2} p_{3}}$ &  $16x_1+16x_2+21x_3-8$  & --24 & --8 & 13 & --3 & --29 & --13 & 3 & --13\\
 \hline
  $\overline{p_{2} p_{23} p_{12} p_{3}}$ &  $16x_1+20x_2+25x_3-12$  & --32 & --16 & 9 & --7 & --37 & --21 & --1 & --17\\
 \hline
 $\overline{p_{1}p_{13} p_{12} p_{3}}$ &  $20x_1+16x_2+25x_3-12$  & --28 & --8 & 17 & --5 & --37 & --17 & --1 & --21\\
 \hline
 $\overline{p_{3}p_{13}p_{23}p_{12}}$ &  $20x_1+20x_2+29x_3-16$  & --36 & --16 & 13 & --7 & --45 & --25 & --5 & --25\\
 \hline
 $\overline{p_{12}p_{13}p_{23} p_{123}}$ &  $28x_1+28x_2+37x_3-32$  & --60 & --32 & 5 & --23 & --69 & --41 & --13 & --41\\
 \hline
 \end{tabular}
}
\end{table}

The $\mathbb{Z}^3$-translation invariance of $\mathrm{Del}_v $ implies that all 3-cells of $\mathrm{Del}_v $ are simplices with Euclidean volume $\frac{1}{6}$. The restriction of $\mathrm{Del}_v $ on the hyperplanes given by $x_1+x_2+ x_3= \tilde{d}, x_i = 0, i=1,2,3$ respectively induces subdivisions on these hyperplanes. Therefore the restriction of $\mathrm{Del}_v $ on  $\Delta_{\tilde{d}} $,  for any $\tilde{d}\geq 1$, is a subdivision on $\Delta_{\tilde{d}} $. Since the Euclidean  volume of $\Delta_{\tilde{d}} $ is $\frac{\tilde{d}^3}{6}$, $\Delta_{\tilde{d}} $ contains $\tilde{d}^3$ simplices.
\end{proof}
\begin{figure}\label{pic6}
{\footnotesize
\caption{Decomposition $\mathrm{Del}_v $ on $[0,1]^3$.}}
 \begin{center}
  \setlength{\unitlength}{0.5cm}
  \begin{picture}(6,6)(0,0)
\put(0,1){\line(0,1){3}}
\put(0,1){\line(2,-1){2}}
\put(0,4){\line(2,-1){2}}
\put(0,1){\line(1,1){2}}
\put(2,0){\line(0,1){3}}
\put(2,0){\line(2,1){2}}
\put(2,3){\line(2,1){2}}
\put(0,4){\line(2,1){2}}
\put(4,1){\line(0,1){3}}
\put(2,5){\line(2,-1){2}}
\put(4,1){\line(-1,1){2}}
\put(0,4){\line(1,0){4}}
\end{picture}
 \setlength{\unitlength}{0.5cm}
  \begin{picture}(6,6)(0,0)
\put(0,1){\line(0,1){3}}
\put(0,1){\line(1,0){4}}
\put(0,4){\line(2,-1){2}}
\put(0,1){\line(1,1){2}}
\put(2,3){\line(2,1){2}}
\put(0,4){\line(2,1){2}}
\put(4,1){\line(0,1){3}}
\put(2,5){\line(2,-1){2}}
\put(4,1){\line(-1,1){2}}
\put(0,4){\line(1,0){4}}
\end{picture}
 \setlength{\unitlength}{0.5cm}
  \begin{picture}(6,6)(0,0)
\put(0,1){\line(0,1){3}}
\put(0,4){\line(2,-1){2}}
\put(0,1){\line(1,1){2}}
\put(2,3){\line(2,1){2}}
\put(0,4){\line(2,1){2}}
\put(4,1){\line(0,1){3}}
\put(2,5){\line(2,-1){2}}
\put(4,1){\line(-1,1){2}}
\put(0,4){\line(1,0){4}}
\put(2,3){\line(0,-1){1}}
\put(2,2){\line(-2,-1){2}}
\put(2,2){\line(2,-1){2}}
\end{picture}
 \setlength{\unitlength}{0.5cm}
  \begin{picture}(6,6)(0,0)
\put(0,1){\line(0,1){3}}
\put(0,4){\line(2,-1){2}}
\put(0,1){\line(1,1){2}}
\put(2,3){\line(2,1){2}}
\put(0,4){\line(2,1){2}}
\put(2,5){\line(2,-1){2}}
\put(0,4){\line(1,0){4}}
\put(2,3){\line(0,-1){1}}
\put(2,2){\line(-2,-1){2}}
\put(2,2){\line(1,1){2}}
\end{picture}
\setlength{\unitlength}{0.5cm}
  \begin{picture}(6,6)(0,0)
\put(0,4){\line(2,-1){2}}
\put(2,3){\line(2,1){2}}
\put(0,4){\line(2,1){2}}
\put(2,5){\line(2,-1){2}}
\put(0,4){\line(1,0){4}}
\put(2,3){\line(0,-1){1}}
\put(2,2){\line(-1,1){2}}
\put(2,2){\line(1,1){2}}
\end{picture}
\setlength{\unitlength}{0.5cm}
  \begin{picture}(6,6)(0,0)
\put(0,4){\line(2,1){2}}
\put(2,5){\line(2,-1){2}}
\put(0,4){\line(1,0){4}}
\put(2,2){\line(-1,1){2}}
\put(2,2){\line(1,1){2}}
\end{picture}
 \end{center}
 \end{figure}

By using this lemma, we  obtain the   triangulation $\mathcal{T}_v$ on $\Delta_d$,  and the restriction of $\mathcal{T}_v$ on  $\Delta_4$ is the restriction of   $\mathrm{Del}_v $  on $\Delta_4$.  The $\mathbb{Z}^3$-translation invariance of $\mathrm{Del}_v $ implies that
for any $m \in \Delta_d^o(\mathbb{Z})$, the restriction of $\mathcal{T}_v$ on $m-(1,1,1)+\Delta_4$ is the translation of the restriction of  $\mathcal{T}_v$ on $\Delta_4$.

 Again  $\Pi_v$ denotes  the tropical variety of the restriction of $v $ on $\Delta_d$, i.e. the non-linear locus of  the  discrete  Legendre transform (\ref{eq4.2}) of $v|_{\Delta_d}$,  which is  dual to $\mathcal{T}_v$. The restriction of $v $ on any  $m-(1,1,1)+\Delta_4$ induces the same tropical variety $\Pi_{K3}$ dual to the restriction of $\mathcal{T}_v$ on $\Delta_4$ by the $\mathbb{Z}^3$-translation.

  \begin{figure}\label{pic7}
{\footnotesize
\caption{Illustration of $\Pi_v$ vs $\Pi_{K3}$ via cases of curves, where $\Pi_1$ is dual  to a  triangulation $\mathcal{T}$  of $\Delta_4\subset \mathbb{R}^2$, and $\Pi_2$ is dual to the restriction of $\mathcal{T}$ on $\Delta_3$. } }
 \begin{center}
  \setlength{\unitlength}{0.5cm}
  \begin{picture}(10,8)(-1,-1)
\put(0,0){\line(1,1){1}}
\put(0,0){\line(-1,0){1}}
\put(0,0){\line(0,-1){1}}
\put(1,1){\line(1,0){1}}
\put(1,1){\line(0,1){1}}
\put(2,1){\line(1,1){1}}
\put(2,1){\line(0,-1){1}}
\put(1,2){\line(-1,0){1}}
\put(1,2){\line(1,1){1}}
\put(3,2){\line(1,0){1}}
\put(3,2){\line(0,1){1}}
\put(2,3){\line(1,0){1}}
\put(2,3){\line(0,1){1}}
\put(3,2){\line(1,0){1}}

\put(3,3){\line(1,1){1}}
\put(4,4){\line(0,1){1}}
\put(4,4){\line(1,0){1}}
\put(4,2){\line(0,-1){1}}
\put(4,2){\line(1,1){1}}
\put(5,3){\line(1,0){1}}
\put(5,3){\line(0,1){1}}
\put(5,4){\line(1,1){1}}
\put(6,3){\line(1,1){1}}
\put(6,3){\line(0,-1){1}}
\put(2,4){\line(-1,0){1}}
\put(2,4){\line(1,1){1}}
\put(3,5){\line(0,1){1}}
\put(3,5){\line(1,0){1}}
\put(4,5){\line(1,1){1}}
\put(3,6){\line(1,1){1}}
\put(3,6){\line(-1,0){1}}
\put(7,0){$\Longleftarrow$}
\put(1,-1){$\Pi_1$}
\end{picture}
\setlength{\unitlength}{0.5cm}
  \begin{picture}(10,8)(-1,-1)
\put(0,0){\line(1,1){1}}
\put(0,0){\line(-1,0){1}}
\put(0,0){\line(0,-1){1}}
\put(1,1){\line(1,0){1}}
\put(1,1){\line(0,1){1}}
\put(2,1){\line(1,1){1}}
\put(2,1){\line(0,-1){1}}
\put(1,2){\line(-1,0){1}}
\put(1,2){\line(1,1){1}}
\put(3,2){\line(1,0){1}}
\put(3,2){\line(0,1){1}}
\put(2,3){\line(1,0){1}}
\put(2,3){\line(0,1){1}}
\put(3,2){\line(1,0){1}}
\put(3,3){\line(1,1){1}}
\put(4,2){\line(0,-1){1}}
\put(4,2){\line(1,1){1}}
\put(2,4){\line(-1,0){1}}
\put(2,4){\line(1,1){1}}
\put(1,-1){$\Pi_2$}
\end{picture}
\end{center}
\end{figure}
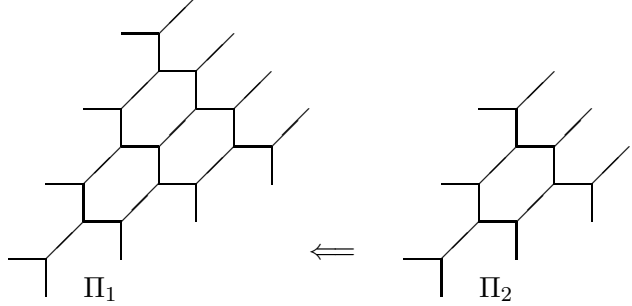

\begin{lemma}\label{le2}
 Let $\mathcal{T}^o$ be the set of 3-cells $\rho \in \mathcal{T}_v$ such that either $\rho$ belongs to $\Delta^o_d$, i.e. $\rho\subset\Delta^o_d$, or $\rho$ shares a 2-face $\rho'$ with the boundary $\partial\Delta^o_d$, i.e. $\rho' \in \mathcal{T}_v$, $\rho' = \rho \cap \partial\Delta^o_d$, and $\dim \rho'=2$.
   Then $\mathcal{T}^o$ consists of $d(d-4)^2$ simplices,  i.e.  $$ \sharp \mathcal{T}^o =d(d-4)^2. $$
\end{lemma}

\begin{proof}
Note that  $\Delta^o_d=(1,1,1)+\Delta_{d-4}$, and the boundary  $\partial\Delta_{d}^o=(1,1,1)+\partial\Delta_{d-4}$ is the union of some 2-cells in $\mathcal{T}_v$.
The restriction of $\mathcal{T}_v$ on $(1,1,1)+\Delta_{d-4}$  consists of $(d-4)^3$ simplices of dimension 3, and $\partial \Delta_{d-4}$ has $4(d-4)^2$ 2-cells of  $\mathcal{T}_v$.  For a 2-cell $\rho' \subset \partial\Delta^o_d$, there is a unique 3-cell $\rho $ in $\mathcal{T}_v$ such that $\rho\supset \rho'$ and $\rho$ does not belong to  $\Delta^o_d$.   Therefore
$$\sharp \mathcal{T}^o=(d-4)^3+4(d-4)^2 = d(d-4)^2.$$
\end{proof}

\begin{lemma}\label{le3}
For any $m\in \Delta^o_d(\mathbb{Z})$,  let $\mathcal{T}^m$ be the set of 3-cells $\rho\in \mathcal{T}_v$ such that $\rho
\subset  m-(1,1,1)+\Delta_4$. Then   $$\bigcup_{m\in \Delta^o_d(\mathbb{Z})} \{(1,1,1)-m+\rho | \rho \in \mathcal{T}^m\backslash \mathcal{T}^o\} = \{3-{\rm cells}  \  \rho\in \mathcal{T}_v| \rho\subset \Delta_4\}.$$
\end{lemma}

\begin{proof}
Note that $$\Delta_4=\Delta_3 \cup ((1,0,0)+\Delta_3)\cup ((0,1,0)+\Delta_3)\cup ((0,0,1)+\Delta_3),$$ and if a 3-cell $\rho\in \mathcal{T}_v$ belongs to $$\Delta_3 \cup ((d-3,0,0)+\Delta_3)\cup ((0,d-3,0)+\Delta_3)\cup ((0,0,d-3)+\Delta_3) , $$ then $\rho $ shares at most one  point with $\Delta^o_d(\mathbb{Z})$. For example, if $\rho\subset \Delta_3 $, then for any $x\in \rho$, $x_1+x_2+x_3\leq 3$, while for an $x \in \Delta^o_d$, $x_1+x_2+x_3> 3$ except $x=(1,1,1)$.  Therefore,
 $\rho$ does not belong to $\mathcal{T}^o$.

It is clear that $(1,1,1)-m+\rho \subset \Delta_4 $ for any  $ \rho \in \mathcal{T}^m$.
If $\rho$ is a 3-cell of $ \mathcal{T}_v$ and $\rho\subset\Delta_4$, then $\rho$ belongs to one of $\Delta_3 $, $ (1,0,0)+\Delta_3$, $ (0,1,0)+\Delta_3$ or $ (0,0,1)+\Delta_3$. For example, if $\rho \subset (1,0,0)+\Delta_3$ without loss of generality, then $(d-4,0,0)+\rho \subset (d-3,0,0)+\Delta_3$.
 We obtain the conclusion by the fact that $\Delta_3\subset  \Delta_4$, $ (d-3,0,0)+\Delta_3\subset (d-4,0,0)+ \Delta_4$, $ (0,d-3,0)+\Delta_3\subset (0,d-4,0)+ \Delta_4$ and $ (0,0,d-3)+\Delta_3\subset  (0,0, d-4)+ \Delta_4$.
\end{proof}

  \begin{figure}\label{pic7}
{\footnotesize
\caption{Illustrations of Lemma \ref{le2} and Lemma  \ref{le3} via  2-dimensional cases. } }
 \begin{center}
  \setlength{\unitlength}{0.5cm}
  \begin{picture}(10,8)(0,0)
\put(0,0){\line(1,0){7}}
\put(0,0){\line(0,1){7}}
\put(7,0){\line(-1,1){7}}
\put(0,1){\line(1,0){6}}
\put(1,0){\line(0,1){6}}
\put(6,0){\line(-1,1){6}}
\put(0,2){\line(1,0){5}}
\put(2,0){\line(0,1){5}}
\put(5,0){\line(-1,1){5}}
\put(0,3){\line(1,0){4}}
\put(3,0){\line(0,1){4}}
\put(4,0){\line(-1,1){4}}
\put(4,0){\line(0,1){3}}
\put(5,0){\line(0,1){2}}
\put(6,0){\line(0,1){1}}
\put(0,4){\line(1,0){3}}
\put(0,5){\line(1,0){2}}
\put(0,6){\line(1,0){1}}
\put(3,0){\line(-1,1){3}}
\put(2,0){\line(-1,1){2}}
\put(1,0){\line(-1,1){1}}
\put(5,5){$\Delta_d$}
\put(1.2,1.2){\circle*{0.2}}
\put(2.2,1.2){\circle*{0.2}}
\put(3.2,1.2){\circle*{0.2}}
\put(4.2,1.2){\circle*{0.2}}
\put(1.2,2.2){\circle*{0.2}}
\put(1.2,3.2){\circle*{0.2}}
\put(1.2,4.2){\circle*{0.2}}
\put(2.2,2.2){\circle*{0.2}}
\put(2.2,3.2){\circle*{0.2}}
\put(3.2,2.2){\circle*{0.2}}
\put(1.7,1.7){\circle*{0.2}}
\put(0.7,1.7){\circle*{0.2}}
\put(2.7,1.7){\circle*{0.2}}
\put(3.7,1.7){\circle*{0.2}}
\put(4.7,1.7){\circle*{0.2}}
\put(2.7,0.7){\circle*{0.2}}
\put(1.7,0.7){\circle*{0.2}}
\put(3.7,0.7){\circle*{0.2}}
\put(4.7,0.7){\circle*{0.2}}
\put(0.7,2.7){\circle*{0.2}}
\put(3.7,2.7){\circle*{0.2}}
\put(1.7,2.7){\circle*{0.2}}
\put(2.7,2.7){\circle*{0.2}}
\put(0.7,3.7){\circle*{0.2}}
\put(1.7,3.7){\circle*{0.2}}
\put(2.7,3.7){\circle*{0.2}}
\put(0.7,4.7){\circle*{0.2}}
\put(1.7,4.7){\circle*{0.2}}
\end{picture}
  \setlength{\unitlength}{0.5cm}
  \begin{picture}(10,8)(0,0)
\put(0,0){\line(1,0){7}}
\put(0,0){\line(0,1){7}}
\put(7,0){\line(-1,1){7}}
\put(1,1){\line(1,0){4}}
\put(1,1){\line(0,1){4}}
\put(5,1){\line(-1,1){4}}
\put(3,0){\line(-1,1){3}}
\put(2,0){\line(-1,1){2}}
\put(5,0){\line(0,1){2}}
\put(0,5){\line(1,0){2}}
\put(1.5,2){$\Delta_d^o$}
\put(1,1){\line(-1,0){1}}
\put(1,1){\line(0,-1){1}}
\end{picture}
\end{center}
\end{figure}
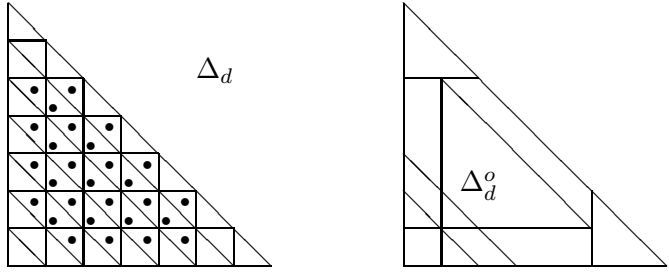

\begin{proof}[Proofs  of i), ii),  iii), and v) in  Theorem   \ref{mainthm}]
First, we recall the construction in Section 4 of \cite{Mik}.
 Let $\overline{\mathcal{P}}^2=\mathbb{CP}^2\backslash (\tilde{H}_1 \cup \tilde{H}_2 \cup \tilde{H}_3 \cup \tilde{H}_4)$, where $ \tilde{H}_i$ denotes a tubular neighbourhood of the hyperplane $H_i$. The interior ${\rm int} (\overline{\mathcal{P}}^2)$ is diffeomorphic to    $\mathcal{P}^2$. The boundary
 $\partial \overline{\mathcal{P}}^2$ admits a stratification $\partial \overline{\mathcal{P}}^2= \partial_0 \mathcal{P}^2 \cup \partial_1 \mathcal{P}^2$ by Proposition 2.24 of \cite{Mik}, where $\partial_0 \mathcal{P}^2$ is the disjoint  union of 6 copies of $T^2$, and $\partial_1 \mathcal{P}^2$ consists of 4 connected components where each component is diffeomorphic to the total space of a  trivial $S^1$-bundle on the 2-dimensional pair-of-pants $\mathcal{P}^1$, i.e. $S^1 \times \mathcal{P}^1$ under a certain trivialisation.    Moreover $\partial_0 \mathcal{P}^2\subset \overline{ \partial_1 \mathcal{P}^2}$.
  The boundary $\partial\overline{\mathcal{P}}^2$ is obtained by gluing the closures of components  of $\partial_1 \mathcal{P}^2$ along  $T^2$  components of  $\partial_0 \mathcal{P}^2$.

Proposition 4.6 in \cite{Mik} shows that there exists a proper submanifold $Q^n \subset (\mathbb{C}^*)^{n+1}$, $n=0, 1, 2$,   such that
 \begin{enumerate}
\item[i)] $Q^0$ is a point.
\item[ii)] $Q^n$ is  isotopic to the hyperplane $$ H^o=\{(w_1,\cdots,w_{n+1})\in (\mathbb{C}^*)^{n+1}|w_1+\cdots+w_{n+1}=0\} \subset (\mathbb{C}^*)^{n+1}.$$ Note that the  pair-of-pants $\mathcal{P}^n$ is bi-holomorphic to  $H^o$.
\item[iii)]  $Q^n$ is invariant under the symmetric group $\mathbb{S}_{n+2}$-action on $(\mathbb{C}^*)^{n+1}$, which interchanges the homogeneous coordinates on $\mathbb{CP}^{n+1}\supset (\mathbb{C}^*)^{n+1}$.
\item[iv)] For a  $\varrho \ll - 1$,  $$Q^n \cap  (\mathbb{C}^*)^{n+1}_{\varrho}=Q^{n-1} \times \mathbb{C}^*_{\varrho},  $$ where $(\mathbb{C}^*)^{n+1}_{\varrho}=\{(w_1, \cdots, w_{n+1})\in (\mathbb{C}^*)^{n+1}| \log |w_{n+1}|< \varrho\}$ and $\mathbb{C}^*_{\varrho}=\{w \in \mathbb{C}^*| \log |w|< \varrho\}$.   Furthermore, $Q^n \cap  (\mathbb{C}^*)^{n+1}_{\varrho}$ is invariant under the translation $w_{n+1} \mapsto c w_{n+1}$ for $0<c<1$.
    \item[v)] If  $\bar{Q}^n=Q^n \cap {\rm Log}^{-1}(R\Delta)$ where $$R\Delta=\{(x_1, \cdots, x_{n+1})\in \mathbb{R}^{n+1}| x_i \geq -R, \ \ i=1,\cdots, n+1, \ \ \sum_{i=1}^{n+1} x_i\leq R\}, $$  for $ R\gg 1$,  then $\bar{Q}^n$ is diffeomorphic to $\overline{\mathcal{P}}^n$ as manifolds with corners.
\end{enumerate}

For each 3-cell $q \in \mathcal{T}_v$, we take a copy $\bar{Q}_q^2$ of $\overline{\mathcal{P}}^2$,  which is identified with  $\bar{Q}^2\subset (\mathbb{C}^*)^3$.  $Q_q^2$ denotes  the interior of $\bar{Q}_q^2$, and $\check{q}$ is  the vertex of $ \Pi_v$ corresponding to $q$. Each connected component of $\partial_1\bar{Q}_q^2= \partial_1 \mathcal{P}^2$ corresponds to a 2-cell of $\mathcal{T}_v$, which is a face of   $q$. If $\rho$ is a 2-cell of $\mathcal{T}_v$ such that  $\rho = q\cap q'$, we consider the corresponding components $F_q$ and $F_{q'} $ of $\partial_1\bar{Q}_q^2$ and $\partial_1\bar{Q}_{q'}^2$ respectively. We glue $\bar{Q}_{q}^2$ and $\bar{Q}_{q'}^2$ along the closures of $F_q$ and $F_{q'} $. More precisely, if we assume that both $F_q$ and $F_{q'} $ are given by $\log |w_3|=-R$ via certain transforms of  the symmetry group $\mathbb{S}_4$, then we attach $F_q$ and $F_{q'} $ by the map $(w_1, w_2, w_3)\mapsto (w_1, w_2, \bar{w}_3)$, i.e. gluing the 2-dimensional  pair-of-pants canonically, and reversing the orientation of the $S^1$-fibre.
In such way, we obtain a manifold $\mathcal{Q}$ with boundary, i.e. $$\mathcal{Q}=\bigcup_{{\rm all \ 3-cells} \ q\in \mathcal{T}_v} \bar{Q}_{q}^2.$$
The boundary $\partial \mathcal{Q}$ is obtained by gluing the closures of connected components $F$ of $\partial_1\bar{Q}_q^2$ corresponding to 2-cells of $\mathcal{T}_v$ in  $\partial\Delta_d$.
 Note that any  two $F$ and $F'$ are  glued along a  certain   $T^2$ component. Thus we have stratified structure $\partial \mathcal{Q}=\partial_0 \mathcal{Q}\cup \partial_1 \mathcal{Q}$ where $\partial_0 \mathcal{Q}$ consists of finite $T^2$, and each connected component of $\partial_1 \mathcal{Q}$ is a $S^1$-bundle.

Let $W^o=\mathcal{Q} \backslash \partial \mathcal{Q}$, and let $W$ be the space obtained  by collapsing the $S^1$ fibres of $\partial_1 \mathcal{Q}$  (cf. Section 4 in \cite{Mik}). Note that $W$ is a differential manifold since the collapsing locally coincides with collapsing the boundary of $\bar{\mathcal{P}}^2$ in $\mathbb{CP}^2$.
 Theorem 4 of \cite{Mik} proves that $W^o$ is diffeomorphic to $X_t^o\subset X_t$, and $W$ is diffeomorphic to
  $X_t$.  Thus $X_t$ is decomposed into $d^3$ copies of pair-of-pants.

Let $$M^o=\coprod_{q\in \mathcal{T}^o}Q_q^2,$$ which is the disjoint union of $d(d-4)^2$ copies of pair-of-pants $\mathcal{P}^2$ by Lemma \ref{le2}. Note that $q$  intersects with the boundary $\partial \Delta_d$ at most one point, i.e. a vertex of $q$, which  corresponds to a connected component of the complement of  $\Pi_v$ in $\mathbb{R}^3$. Therefore any non-zero  dimensional  face $\rho \subset q$ does not belong to the boundary $\partial \Delta_d$, and  the closure $\overline{M^o}\subset W^o$. Furthermore, $$  \overline{M^o}\backslash M^o=\bigcup_{q\in \mathcal{T}^o}\partial\bar{Q}_q^2.$$

\begin{figure}\label{pic--+++}
{\footnotesize
\caption{$\bar{Q}^1$, and $\mathcal{P}^1=Q^1\rightarrow \mathrm{Y}$.}}
 \begin{center}
  \setlength{\unitlength}{0.5cm}
  \begin{picture}(8,7)(0,0)
  \put(5.5,4.5){\line(1,1){1.5}}
\put(1,4.5){\line(1,0){2}}
\put(1,3.5){\line(1,0){2}}
\put(4,1){\line(0,1){1.5}}
\put(5,1){\line(0,1){1.5}}
\put(4.5,5){\line(1,1){1.5}}
\qbezier(3,4.5)(4,4.5)(4.5,5)
\qbezier(3,3.5)(4,3.5)(4,2.5)
\qbezier(5,2.5)(5,4)(5.5,4.5)
\qbezier(1,4.5)(1.5,4)(1,3.5)
\qbezier(1,4.5)(0.5,4)(1,3.5)
\qbezier(4,1)(4.5,1.5)(5,1)
\qbezier(4,1)(4.5,0.5)(5,1)
\qbezier(7,6)(6.7,6.7)(6,6.5)
\put(1.2,4){\line(1,0){3.3}}
\put(4.5,4){\line(0,-1){2.7}}
\put(4.5,4){\line(1,1){2.3}}
   \put(2,2){$\bar{Q}^1$}
\end{picture}
  \setlength{\unitlength}{0.5cm}
  \begin{picture}(8,7)(0,0)
\put(6,4){\circle{2}}
\put(5,4){\line(1,0){2}}
\put(3.5,4){$ =$}
\qbezier(2,4)(0,6)(1,4)
\qbezier(2,4)(0,2)(1,4)
\qbezier(2,4)(4,3)(1,4)
\end{picture}
 \end{center}
 \end{figure}
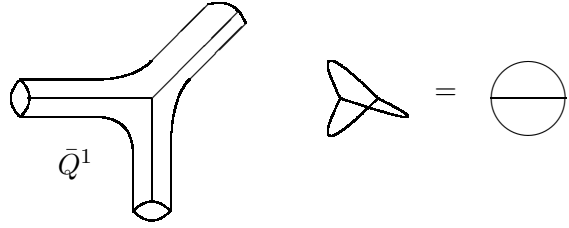

Note that there is a fibration $\mathcal{P}^1\rightarrow \mathrm{Y}$ from $\mathcal{P}^1$ to a  graph of $\mathrm{Y}$-shape with generic fibres  $S^1$ and one  singular fibre of shape $\ominus$.
 For each $\bar{Q}_q^2$, any component $F$ of $\partial_1 \bar{Q}_q^2$ admits a fibration $F \rightarrow \mathrm{Y}$ to a $Y$-shape graph $\mathrm{Y}$ with generic fibres  $T^2$ and one singular fibre of shape $\ominus \times \bigcirc$. Since $\partial_0 \bar{Q}_q^2$ consists of  6 copies of $T^2$, we have a fibration $\partial \bar{Q}_q^2 \rightarrow B_q$ where $B_q$ is a graph consisting of  4 vertices and 6 edges with generic fibres  $T^2$ and 6 singular fibres  $\ominus \times \bigcirc$. If we have glued $\bar{Q}_{q}^2$ and $\bar{Q}_{q'}^2$ along the components  $F_q$ and $F_{q'} $ of $\partial_1 \bar{Q}_q$ and $\partial_1 \bar{Q}_{q'}$ respectively, then we glue the closures of the   corresponding graphs $\mathrm{Y}_q$ and $\mathrm{Y}_{q'}$, and obtain a graph $B$. Moreover, there is a fibration $ \overline{M^o}\backslash M^o \rightarrow B$ with generic fibres $T^2$ and finite singular fibres of shape $\ominus \times \bigcirc$.

 We have proved i) and ii) of Theorem   \ref{mainthm}, and now we prove the statement iii).

 We apply the above construction to the restriction of $\mathcal{T}_v$ on $\Delta_4$ and $\Pi_{K3}$, and obtain a  manifold $Y$ diffeomorphic to the K3 surface. More precisely, for any $m\in \Delta^o_d(\mathbb{Z})$, we consider $m-(1,1,1)+\Delta_4$. If we  let $$\mathcal{Q}_m=\bigcup_{q\in\mathcal{T}^m} \bar{Q}_{q}^2\subset \mathcal{Q},$$ where $\mathcal{T}^m$ is defined in Lemma \ref{le3}, then as  above,  $\mathcal{Q}_m$ is a manifold with boundary $\partial \mathcal{Q}_m$, which is obtained by gluing the closures of connected components $F$ of $\partial_1\bar{Q}_q^2$ corresponding to 2-cells in  $m-(1,1,1)+\partial \Delta_4$. Since the restriction of $\mathcal{T}_v$ on $m-(1,1,1)+\Delta_4$ is the translated to the  restriction of $\mathcal{T}_v$ on $\Delta_4$, $\mathcal{Q}_m$ is independent of the choice of $m$, and  we obtain $Y$ by collapsing the boundary $\partial \mathcal{Q}_m$ as the above construction. And $Y$ is a K3 surface.

  If  $$\mathcal{Q}_m'=\bigcup_{q\in\mathcal{T}^m\backslash \mathcal{T}^o} \bar{Q}_{q}^2\subset \mathcal{Q}_m,$$
then we collapse the closures of the components of $\partial_1\bar{Q}_{q}^2$ corresponding to 2-cells in $\partial \Delta_d \cap \partial (m-(1,1,1)+\Delta_4)$.   Note that if $q$ shares only one  2-cell (or two 2-cells) with $\partial \Delta_d$, then the interior of the collapsed  $\bar{Q}_{q}$ is $(\mathbb{C}^*)^2$ ($\mathbb{C}\times \mathbb{C}^* $ respectively).
We obtain a manifold $Y_m$ with boundary, and denote the interior as $Y_m^o=Y_m\backslash \partial Y_m$,  which satisfies $$Y_m^o \supset  \coprod_{q\in \mathcal{T}^m\backslash \mathcal{T}^o}Q_q^2.$$   Then $Y_m^o$
 can be regarded as a subset both in $Y$ and $W$, i.e. $Y_m^o \subset Y$ and $Y_m^o\subset W$, and can have many connected components including   $(\mathbb{C}^*)^2$ and $\mathbb{C}\times \mathbb{C}^* $.

  Since  $\bigcup\limits_{m\in \Delta^o_d(\mathbb{Z})}(\mathcal{T}^m \backslash \mathcal{T}^o) \bigcup  \mathcal{T}^o $ consists of all 3-cells of $\mathcal{T}_v$,  $$ W=\overline{M^o} \bigcup \bigcup_{m\in \Delta^o_d(\mathbb{Z})} \overline{Y_m^o}, \  \ \ {\rm and} \ \  \   Y=\bigcup_{m\in \Delta^o_d(\mathbb{Z})} \overline{Y_m^o}, $$
  by regarding $Y_m^o\subset Y$ and Lemma \ref{le3}. Note that $Y_m^o$ does not intersect with $Q_{q}^2$, for any $m\in \Delta^o_d(\mathbb{Z})$ and $q\in  \mathcal{T}^o$, when they are  regarded as subsets in $W$, and thus $Y_m^o\cap M^o $ is empty.
    We obtain iii) by choosing an open dense  subset $M'\subset Y$ such that any connected component of $M'$ belongs to a certain $Y_m^o$.

 Finally the Riemann-Roch formula shows that   $$2\chi(X)+3\tau(X)=K_X^2=d(d-4)^2,$$ and  we obtain v)  by (\ref{eq3.1}).
\end{proof}

We remark that many arguments in this section may be carried out in the more general frame work, the Gross-Siebert programme  (cf. \cite{GS1,GS2}). We expect to generalise  Theorem \ref{mainthm}  to more general cases, and leave it to the future research.

\section{Degenerations}
The final  section  finishes  the proof of Theorem \ref{mainthm}.
We are continuing to use the notations  and the conventions  in  Section 4.

 Recall that $X_t$ is the patchworking family given by the polynomial \begin{equation}\label{eq5.1}f_t(w)=\sum_{m\in \Delta_d(\mathbb{Z})} t^{-v(m)}w^m,  \  \   \  \  t\in (1, \infty), \end{equation} where $v$ is defined by (\ref{eq4.3.1}), which is a family of degree $d$ hypersurfaces in $\mathbb{CP}^3$.
If  $[z_0,z_1,z_2,z_3]$ are  the homogeneous coordinates on $\mathbb{CP}^3$, then $w_i=z_i/z_0$, $i=1,2,3$.
  Note that the canonical bundle  $\mathcal{O}_{X_t}(K_{X_t})=\mathcal{O}_{\mathbb{CP}^3}(d-4)|_{X_t}$ by the adjunction formula,  and is very ample.

    For any $m=(m_1, m_2, m_3)\in \Delta^o(\mathbb{Z})$, since   $m_1+m_2+m_3\leq d-1$, and $m_i\geq 1$, $i=1,2,3$,   $$t^{-v(m)}\frac{w^m dw_1 \wedge dw_2 \wedge dw_3}{w^{(1,1,1)}f_t}= t^{-v(m)}\frac{w_1^{m_1-1}w_2^{m_2-1}w_3^{m_3-1} dw_1 \wedge dw_2 \wedge dw_3}{f_t}$$ is a meromorphic 3-form on $\mathbb{CP}^3$,  and has a simple pole along $X_t$. The Poincar\'{e} residue formula gives a holomorphic 2-form $$\Omega_m=t^{-v(m)}{\rm Res}_{X_t}\frac{w^m dw_1 \wedge dw_2 \wedge dw_3}{w^{(1,1,1)}f_t}$$ on $X_t$, i.e. $$\Omega_m=t^{-v(m)}\frac{w^{m}dw_2\wedge dw_3}{w^{(1,1,1)}\frac{\partial f_t}{\partial w_1}}=-t^{-v(m)}\frac{w^{m}dw_1\wedge dw_3}{w^{(1,1,1)}\frac{\partial f_t}{\partial w_2}}=t^{-v(m)}\frac{w^{m}dw_1\wedge dw_2}{w^{(1,1,1)}
  \frac{\partial f_t}{\partial w_3}}$$ on $X_t$.  Note that $\Omega_m$ represents a non-trivial cohomological  class, i.e. $$0\neq [\Omega_m]\in H^2(X_t, \mathbb{C}).$$

\subsection{Proof  of iv)  in  Theorem   \ref{mainthm}}
Let $\rho $ be a 3-cell in $\mathcal{T}_v$ such that $\rho \cap \Delta_d^o(\mathbb{Z})$ is not empty, and let $m\in \rho \cap \Delta_d^o(\mathbb{Z})$. If $m' \in \rho \cap \Delta_d(\mathbb{Z})$ is a different vertex of $\rho$, i.e. $m\neq m'$, then the 1-cell $\overline{m,m'} \subset \rho$ connecting $m$ and $m'$ corresponds a 2-cell $\Pi_{m,m'}$ in $\Pi_v$. For any $x\in {\rm int}(\Pi_{m,m'})$, $$l_{m}(x)=l_{m'}(x)> l_{m''}(x),$$ for any
$m'' \in \Delta_d(\mathbb{Z})\backslash \overline{m,m'}$, where $l_m(x)=\langle m, x\rangle -v(m)$.  We assume that $m_3-m_3'\neq 0$
without loss of generality.

By  Corollary 6.4 in \cite{Mik}, the  amoebas $\mathcal{A}_t$ converges to $\Pi_v$ when $t\rightarrow \infty$, and furthermore, Theorem 5 of \cite{Mik} asserts that  a certain normalisation of  $X_t^o$ converges to a limit in $(\mathbb{C}^*)^3$.
 We take a close  look at the limit in the current case.

  If  $V$ is  a  small open  neighborhood  of a point in $ {\rm int}(\Pi_{m,m'})$ such that $\overline{V \cap \Pi_v}\subset {\rm int}(\Pi_{m,m'})$, then $V \cap \Pi_v$ belongs to  the hyperplane given by the equation $$l_{m}(x)-l_{m'}(x)=\langle m-m', x\rangle-v(m)+v(m')=0.$$
Let $\theta_i$ be the angle of $w_i$, i.e. $$w_i=\exp (x_i\log(t)+\sqrt{-1}\theta_i), \  \  \  i=1,2,3.$$ We regard $(\theta_1,\theta_2,\theta_3)$ as angular coordinates on $T^3$, and identify  $(\mathbb{C}^*)^3=\mathbb{R}^3\times T^3$ by $w \mapsto (x, \theta)$, where  $x=(x_1, x_2, x_3)$ and  $\theta=(\theta_1, \theta_2, \theta_3)$.   Equivalently, we regards $(\mathbb{C}^*)^3$ as the quotient of $\mathbb{C}^3$ by $\sqrt{-1}\mathbb{Z}^3$.

\begin{lemma}\label{le++5.1}
 If $\mathrm{H}_t: (\mathbb{C}^*)^3\rightarrow \mathbb{R}^3\times T^3$ is defined  by $w \mapsto (x, \theta)$, then  $\mathrm{H}_t(X_t^o)\cap (V\times T^3)$ converges to $W_V$ in the Hausdorff sense, when $t\mapsto \infty$,  where $$ W_V =\{(x, \theta)\in V\times T^3|l_{m}(x)=l_{m'}(x),  \  \  \langle m-m', \theta\rangle=\pi+2\pi\mathbb{Z} \},$$ and  $\langle m, \theta\rangle=m_1\theta_1+m_2\theta_2+m_3\theta_3$.
\end{lemma}

\begin{proof} Note that the log map ${\rm Log}_t$ is the composition of  $\mathrm{H}_t$ with the projection $\mathbb{R}^3\times T^3 \rightarrow \mathbb{R}^3$.
For any $x\in V\cap \mathcal{A}_t$ and $w\in {\rm Log}_t^{-1}(x)\cap X_t^o$, $t^{-v(m)}|w^m|=t^{l_m(x)} $, and
 the patchworking polynomial (\ref{eq5.1}) says that $$ 0=e^{\sqrt{-1}\langle m, \theta\rangle}+t^{l_{m'}(x)-l_m(x)}e^{\sqrt{-1}\langle m', \theta\rangle}+\sum_{m'' \in \Delta_d(\mathbb{Z})\backslash \overline{m,m'}} t^{l_{m''}(x)-l_m(x)}e^{\sqrt{-1}\langle m'', \theta\rangle},$$ which is the defining equation of $ X_t^o \cap (V\times T^3) $.  When $t\rightarrow\infty$, since $ t^{l_{m''}(x)-l_m(x)} \rightarrow 0$, we obtain  $$ e^{\sqrt{-1}\langle m, \theta\rangle}+t^{l_{m'}(x)-l_m(x)}e^{\sqrt{-1}\langle m', \theta\rangle}\rightarrow 0, \  \  {\rm and} \  \ t^{l_{m'}(x)-l_m(x)} \rightarrow 1. $$
 The limiting equations are
  $$ \langle m-m', \theta\rangle=\pi + 2\pi \mathbb{Z},  \  \ {\rm and} \  \ l_{m}=l_{m'}, $$ which define $W_V$.
  If we define a neighbourhood $\Xi_t$ of $W_V$ in $V\times T^3$ by the inequality $$ \Big|e^{\sqrt{-1}\langle m, \theta\rangle}+t^{l_{m'}(x)-l_m(x)}e^{\sqrt{-1}\langle m', \theta\rangle}\Big| < 2 \sup_{x'\in V}\Big( \sum_{m'' \in \Delta_d(\mathbb{Z})\backslash \overline{m,m'}} t^{l_{m''}(x')-l_m(x')}\Big),$$ then $ X_t^o \cap (V\times T^3)\subset  \Xi_t$, and converges to $W_V$ in the Hausdorff sense.
\end{proof}

We have an  isotopic  embedding  $\phi_t:  W_V \rightarrow X_t^o$ with $\phi_\infty={\rm id}$, which is certainly not unique. One way to obtain $\phi_t$ is to integrate a vector field  $(u,\vartheta)$ satisfying $\frac{\partial \tilde{f}_t}{\partial t}+ \langle\frac{\partial \tilde{f}_t}{\partial x}, u\rangle + \langle\frac{\partial \tilde{f}_t}{\partial \theta}, \vartheta\rangle =0,$ where $\tilde{f}_t= t^{-l_m(x)}f_t$, since a direct calculation shows that
$|\frac{\partial \tilde{f}_t}{\partial x_3}|>  \frac{1}{2}|m_3'-m_3|\log(t)$, and $|\frac{\partial \tilde{f}_t}{\partial \theta_3}|>  \frac{1}{2}|m_3'-m_3|$,  for $t\gg 1$.

If    $$T_{m,m'}=\{\theta\in T^3|\langle m-m', \theta\rangle=\pi + 2\pi \mathbb{Z} \},$$ then $$W_V=(V\cap \Pi_v)\times T_{m,m'}.$$

\begin{lemma}\label{le++5.2}
For any  isotopic  embedding  $\phi_t:  W_V \rightarrow X_t^o$ with $\phi_\infty={\rm id}$ and $x\in V\cap \Pi_{m,m'}$,   $$\int_{\phi_t(\{x\}\times T_{m,m'})}\Omega_m  \neq 0, $$ and consequently, $\phi_t(\{x\}\times T_{m,m'})$ represents a non-zero class in
$H_2(X_t, \mathbb{R})$ for $t\gg 1$, i.e. $$0\neq [\phi_t(\{x\}\times T_{m,m'})]\in H_2(X_t, \mathbb{R}).$$ Furthermore, if $\varpi$ is a  toric symplectic form  on $\mathbb{CP}^3$,  then $$\int_{\phi_t(\{x\}\times T_{m,m'})}\varpi=0.$$
\end{lemma}

\begin{proof}
 A direct calculation shows that  \begin{align*}w_3\frac{\partial f_t}{\partial w_3}& =m_3t^{-v(m)}w^m+m_3't^{-v(m')}w^{m'}+\sum_{m'' \in \Delta_d(\mathbb{Z})\backslash \overline{m,m'}}m_3''t^{-v(m'')}w^{m''} \\ & = (m_3-m_3')t^{-v(m)}w^m+m_3'f_t +\sum_{m'' \in \Delta_d(\mathbb{Z})\backslash \overline{m,m'}}(m_3''-m_3')t^{-v(m'')}w^{m''} \\ & =   (m_3-m_3')t^{-v(m)}w^m+m_3'f_t + o(t).\end{align*}
 Note that on $V$, $$ \frac{t^{-v(m'')}|w^{m''}|}{t^{-v(m)}|w^{m}| }=t^{(l_{m''}-l_m)( {\rm Log}_t(w))}\rightarrow 0,$$ and thus $$ \frac{|o(t)|}{t^{-v(m)}|w^{m}| }\rightarrow 0,  \ \ \ {\rm when} \ \ t\rightarrow\infty.$$
  Therefore $$\Omega_m|_{{\rm Log}_t^{-1}(V) \cap X_t}= \frac{ d\log (w_1)\wedge d\log (w_2) }{m_3-m_3'+t^{v(m)}w^{-m}o(t)},$$  and $$\int_{\phi_t(\{x\}\times T_{m,m'})}\Omega_m \rightarrow \int_{T_{m,m'}}\frac{  d\theta_1\wedge d\theta_2 }{m_3'-m_3} \neq 0, \  \ {\rm when} \  \  t\rightarrow\infty.$$  We obtain the first  conclusion.

    Since $(\mathbb{C}^*)^3=\mathbb{R}^3\times T^3 \rightarrow \mathbb{R}^3$ is a Lagrangian fibration with respect to $\varpi$, i.e. $ \varpi|_{\{x\}\times T^3}\equiv0$,  we have $$\int_{\phi_t(\{x\}\times T_{m,m'})}\varpi=\int_{\{x\}\times T_{m,m'}}\varpi =0.$$
\end{proof}

We continue to prove Theorem  \ref{mainthm}.

\begin{proof}[Proof of  iv) in  Theorem   \ref{mainthm}]
Let $Q^2\subset (\mathbb{C}^*)^3$ be the submanifold constructed in Proposition 4.6 of  \cite{Mik}. Recall that the last assertion of Proposition 4.6 of  \cite{Mik} says that for a certain   $\varrho \ll - 1$,  $Q^n \cap  (\mathbb{C}^*)^{n+1}_{\varrho}=Q^{n-1} \times \mathbb{C}^*_{\varrho},  $ where $n=1,2$, and $(\mathbb{C}^*)^{n+1}_{\varrho}=\{(w_1, \cdots, w_{n+1})\in (\mathbb{C}^*)^{n+1}| \log |w_{n+1}|< \varrho\}$,  and $Q^n \cap  (\mathbb{C}^*)^{n+1}_{\varrho}$ is invariant under the translation $w_{n+1} \mapsto c w_{n+1}$ for $0<c<1$.  Hence   \begin{align*}\tilde{Q}_1 &= Q^2 \cap \{(w_1, w_2, w_{3})\in (\mathbb{C}^*)^{3}|
 \log |w_{2}|< \varrho, \log |w_{3}|< \varrho\}\\ & = \{(w_1, w_2, w_{3})\in (\mathbb{C}^*)^{3}|
  \log |w_{2}|< \varrho, \log |w_{3}|< \varrho, w_1\equiv -1\},\end{align*} where we choose  $ w_1\equiv -1$ that   defines $Q^0\subset \mathbb{C}^*$. If we consider the identification $(\mathbb{C}^*)^{3}=\mathbb{R}^3\times T^3$, via $x=(x_1,x_2,x_3)\in \mathbb{R}^3$, $\theta=(\theta_1,\theta_2,\theta_3)\in T^3$, and $w_i=\exp (x_i+\sqrt{-1}\theta_i)$, $i=1,2,3$, then $$\tilde{Q}_1= \{x\in \mathbb{R}^3|x_1=0, x_2 <\varrho, x_3<\varrho\}\times  \{\theta\in T^3|\theta_1=\pi\}.$$

  If $\Pi$ denotes  the tropical hyperplane of the pair-of-pants $H^o$ defined by $1+w_1+w_2+w_3=0$, i.e. the non-smooth locus of $\max\{0,x_1, x_2,x_3\}$,  then $\Pi_1=\{(x_1,x_2,x_3)\in \mathbb{R}^3|x_1=0, x_2\leq 0, x_3\leq 0\} $ is the  2-cell in $\Pi$ corresponding to the 1-cell $ \overline{(0,0,0),(1,0,0)}$ in the standard simplex $\Delta$. Since the  ${\rm Log}$ map  is the projection of $\mathbb{R}^3\times T^3 \rightarrow \mathbb{R}^3$,  ${\rm Log} (\tilde{Q}_1)\subset \Pi_1$,  $${\rm Log}_t (\tilde{Q}_1)=\{x\in \mathbb{R}^3|x_1=0, x_2 <\frac{\varrho}{\log(t)}, x_3<\frac{\varrho}{\log(t)}\}\subset \Pi_1,$$ and  $\bigcup\limits_{t> 1} {\rm Log}_t (\tilde{Q}_1)={\rm int}(\Pi_1)$. Furthermore, if $V$ is an open subset such that
  $\overline{V\cap \Pi_1}\subset {\rm int}(\Pi_1)$, then the same argument as in the proof of Lemma \ref{le++5.1} shows that $\mathrm{H}_t(H^o \cap (V\times T^3))$ converges to $(\Pi_1\cap V)\times \{\theta\in T^3|\theta_1=\pi\}\subset \tilde{Q}_1$ in the Hausdorff sense.

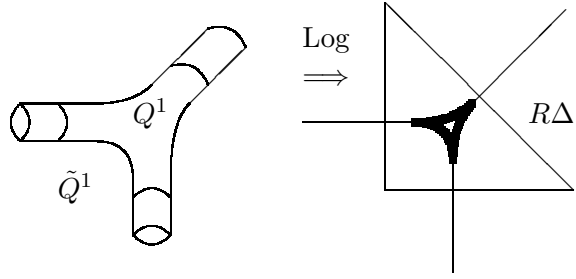
\begin{figure}\label{pic--+}
{\footnotesize
\caption{ Illustrations of $ R\Delta$, $\tilde{Q}^1$, and $Q^1$.}}
 \begin{center}
  \setlength{\unitlength}{0.5cm}
  \begin{picture}(8,8)(0,0)
  \put(5.5,4.5){\line(1,1){1.5}}
\put(1,4.5){\line(1,0){2}}
\put(1,3.5){\line(1,0){2}}
\put(4,1){\line(0,1){1.5}}
\put(5,1){\line(0,1){1.5}}
\put(4.5,5){\line(1,1){1.5}}
\qbezier(3,4.5)(4,4.5)(4.5,5)
\qbezier(3,3.5)(4,3.5)(4,2.5)
\qbezier(5,2.5)(5,4)(5.5,4.5)
\qbezier(2,4.5)(2.5,4)(2,3.5)
\qbezier(1,4.5)(1.5,4)(1,3.5)
\qbezier(1,4.5)(0.5,4)(1,3.5)
\qbezier(4,2)(4.5,2.5)(5,2)
\qbezier(4,1)(4.5,1.5)(5,1)
\qbezier(4,1)(4.5,0.5)(5,1)
\qbezier(6,5)(5.7,5.7)(5,5.5)
\qbezier(7,6)(6.7,6.7)(6,6.5)
\put(4,4){$Q^1$}
   \put(2,2){$\tilde{Q}^1$}
\end{picture}
  \setlength{\unitlength}{0.5cm}
  \begin{picture}(8,8)(0,0)
  \put(0,6){${\rm Log}$}
  \put(0,5){$\Longrightarrow $}
    \put(6,4){$R\Delta$}
\put(2.2,2.2){\line(1,0){5}}
\put(2.2,2.2){\line(0,1){5}}
\put(7.2,2.2){\line(-1,1){5}}
\thinlines
\put(0,4){\line(1,0){3}}
\put(4,0){\line(0,1){3}}
\put(4.5,4.5){\line(1,1){2.5}}
\linethickness{1mm}
\qbezier(3,4)(4,4)(4.5,4.5)
\qbezier(3,4)(4,4)(4,3)
\qbezier(4,3)(4,4)(4.5,4.5)
\end{picture}
 \end{center}
 \end{figure}

    By the $\mathbb{S}_4$-symmetries of $\Pi$ and $Q^2$, there is an open subset $\tilde{Q}^2\subset Q^2$ consisting of 6 connected components, and each component is a copy of $ \tilde{Q}_1$ by passing to the  $\mathbb{S}_4$-action on $(\mathbb{C}^*)^3$.  Moreover, $\bigcup\limits_{t> 1} {\rm Log}_t (\tilde{Q}^2)$ is the union of the interior of   2-cells of $\Pi$.

 Now, we recall the proof of Theorem 4 in \cite{Mik}, which is in Subsection 6.6 of \cite{Mik}. Note that $ASL(3, \mathbb{Z})$  acts on $\mathbb{Z}^3$ by $A(m)_j=\sum a_{ji}m_i+b_j$ for any $A\in ASL(3, \mathbb{Z})$, and acts on $(\mathbb{C}^*)^3$ by $w^m \mapsto w^{A(m)}$.
   For any 3-cell $q$ in $\mathcal{T}_v$,  there is an open neighbourhood $V_q\subset \mathbb{R}^3$ of the vertex $ o\in\Pi$, and an $A_q \in ASL(3, \mathbb{Z})$ such that $\check{q}=A_q(o)$ where $\check{q}\in\Pi_v$ is corresponding vertex of $q$, and  $$\Pi_v=\bigcup_{{\rm all} \ 3-{\rm cells}\ q\in \mathcal{T}_v} A_q(V_q\cap \Pi). $$  Then $$\mathcal{Q}_t^o =\bigcup_{{\rm all} \ 3-{\rm cells}\ q\in \mathcal{T}_v} A_q((V_q\times T^3)\cap \mathrm{H}_t(Q^2))$$ is diffeomorphic to $W^o$, and is isotopic to $X_t^o$ for $t\gg 1$ by Subsection 6.6 of of \cite{Mik}.

    Under the identification $W^o=\mathcal{Q}^o_t$, we have $M^o \subset \mathcal{Q}^o_t$, and $Q^2_q \subset A_q((V_q\times T^3)\cap \mathrm{H}_t(Q^2))$ for any $q\in\mathcal{T}^o$. If $T^2$ is a generic fibre of $\lambda: \overline{M^o}\backslash M^o \rightarrow B$, then $T^2 $ is a generic fibre of  $ \lambda_q: \partial \bar{Q}^2_q= \bar{Q}^2_q\backslash Q^2_q\rightarrow B_q$ for a certain $q\in\mathcal{T}^o$. Recall that $\bar{Q}^2_q=A_q((R\Delta\times T^3)\cap \mathrm{H}_t(Q^2))$ where $R\Delta=\{x\in\mathbb{R}^3| x_i\geq -R, i=1,2,3, x_1+x_2+x_3\leq R\}$ for a certain $R>0$, and $R\Delta \subset V_q$.
    The boundary $\partial \bar{Q}^2_q=A_q((\partial (R\Delta)\times T^3)\cap \mathrm{H}_t(Q^2))$ consists of 4 copies of $S^1$-bundles over  $\bar{\mathcal{P}}^1$ gluing along 6 copies of $T^2$. Note that for any 2-face  $\rho_2 \subset \partial (R\Delta)$, for example given by $ x_1=-R$, the intersection $$
    (\rho_2 \times T^3)\cap \mathrm{H}_t(Q^2) \subset Q^1\times \{w_1\in \mathbb{C^*}| \log |w_1|=-R\}= Q^1\times S^1,$$ and is diffeomorpic to $\bar{\mathcal{P}}^1\times S^1$. The fibration $\lambda_q$ is obtained by gluing $ (\rho_2 \times T^3)\cap \mathrm{H}_t(Q^2) \rightarrow \rho_2 \cap \Pi$ along boundaries, and $ \rho_2 \cap \Pi$ is the $\mathrm{Y}$-shape graph.

    Let $T^2$ be a generic fibre of $\lambda: \overline{M^o}\backslash M^o \rightarrow B$.
     By letting $t\gg 1$, we can take $A_q^{-1}(T^2)$ as a fibre of $(\rho_2 \times T^3)\cap \mathrm{H}_t(\tilde{Q}^2) \rightarrow  \rho_2 \cap {\rm Log}_t(\tilde{Q}^2)$, which is the restriction of the $T^2$-fibration $ \tilde{Q}^2 \rightarrow {\rm Log}_t(\tilde{Q}^2)$.
    By passing to the $\mathbb{S}_4$-action, we can assume that $$A_q^{-1}(T^2)=\{(x, \theta)\in \mathbb{R}^3\times T^3| \theta_3=\pi, x_1=-R, x_2=-\frac{R}{2}, x_3=0\}.$$  We apply Lemma \ref{le++5.2} to $T^2\subset A_q((V_q\times T^3)\cap \mathrm{H}_t(\tilde{Q}^2))$, and obtain the conclusion iv).
\end{proof}

\subsection{Proof  of vi)  in  Theorem   \ref{mainthm}}
   Note that
$\{\Omega_m | m \in \Delta^o(\mathbb{Z})\}$ is a basis  of $H^0(X_t, K_{X_t})$, and defines an embedding \begin{equation}\label{eq5.2}\Psi_t: X_t \rightarrow \mathbb{CP}^N,\ \  \ N=\sharp \Delta^o(\mathbb{Z})-1=p_g-1,\end{equation} such that  $\mathcal{O}_{X_t}(K_{X_t})=\Psi_t^*\mathcal{O}_{\mathbb{CP}^N}(1)$.
If
 \begin{equation}\label{eq5.3}Z_m=t^{-v(m)}w^m, \ \ \ \  \  m \in \Delta^o(\mathbb{Z}),\end{equation}  then we embed $X_t^o$ in $(\mathbb{C}^*)^N\subset \mathbb{CP}^N$ via $Z_m/Z_{(1,1,1)}$, $ m \in \Delta^o(\mathbb{Z})$,  and the closure of  $X_t^o$ in $ \mathbb{CP}^N$ is the image  $\Psi_t(X_t)$. We regard $Z_m$, $ m \in \Delta^o(\mathbb{Z})$,  as
  the homogeneous coordinates of $\mathbb{CP}^N$. Furthermore, $Z_m$ are generalised theta functions in the sense of \cite{GS2}.

  For any 2-cell $\rho_2$ in $\partial \Delta_{d}^o$ of the subdivision $\mathcal{T}_v$, let $$X_{\rho_2}=\{Z_m=0, \ \  m \in \Delta_{d}^o(\mathbb{Z})\backslash \rho_2\}\subset \mathbb{CP}^N,$$ and for any 3-cell $\rho_3 \subset \Delta_{d}^o$, let   \begin{align*}X_{\rho_3} =& \{Z_{m^0}+Z_{m^1}+Z_{m^2}+Z_{m^3}=0, \ \  \{ m^0, m^1, m^2, m^3 \}=  \Delta_{d}^o(\mathbb{Z})\cap \rho_3, \\ & {\rm and} \ Z_m=0,  \ \  m \in \Delta_{d}^o(\mathbb{Z})\backslash \rho_3\}\subset \mathbb{CP}^N. \end{align*}
 We define a singular variety   \begin{equation}\label{eq5.4}X_0=\Big(\bigcup_{{\rm all} \ 3-{\rm cells} \ \rho_3\subset  \Delta_{d}^o} X_{\rho_3} \Big)\bigcup \Big(\bigcup_{{\rm all} \ 2-{\rm cells} \ \rho_2\subset  \partial\Delta_{d}^o}  X_{\rho_2}\Big), \end{equation} which has $d(d-4)^2$ irreducible components by Lemma \ref{le2}, and each component is the complex projective plane, i.e.  $X_\rho = \mathbb{CP}^2$.

\begin{lemma}\label{le4}
The regular locus $X_0^o$ of $X_0$ consists of  $d(d-4)^2$ copies of  pair-of-pants, i.e. $X_0^o=\coprod\limits^{d(d-4)^2}\mathcal{P}^2$.
\end{lemma}

\begin{proof}
If $X_\rho$ is the corresponding component of a 3-cell $\rho\subset \Delta_{d}^o$, then the intersection with another component $X_{\rho'}$ is a hyperplane  $\{Z_{m^0}=0\}\cap X_\rho$ without loss of generality, where $\{m^0, m^1, m^2, m^3\}=\Delta_{d}^o(\mathbb{Z})\cap \rho$ and $\{ m^1, m^2, m^3\}\subset \rho \cap \rho'$.  Thus the regular locus $$X_\rho^o=\{[Z_{m^0},Z_{m^1},Z_{m^2},Z_{m^3}]\in X_\rho|Z_{m^0}Z_{m^1}Z_{m^2}Z_{m^3}\neq 0\}$$ is a pair-of-pants, i.e. $X_\rho^o=\mathcal{P}^2$. More precisely, if we regards $X_\rho$ as the projective plane $\{[Z_{m^1}, Z_{m^2}, Z_{m^3}]\in \mathbb{CP}^2\}$, then $Z_{m^i}=0$, $i=1,2,3$, give two coordinates axis, and the infinite line, and $-Z_{m^0}=Z_{m^1}+Z_{m^2}+Z_{m^3}=0$  defines the fourth line.

If $X_\rho$ corresponds to a 2-cell $\rho\in \partial\Delta_{d}^o$, then the intersection with other three  components $X_{\rho'}$ corresponding three  different  2-cells $\rho'\in \partial\Delta_{d}^o$ is given by $Z_{m^1}Z_{m^2}Z_{m^3}=0$ where $\{m^1, m^2, m^3\}=\Delta_{d}^o(\mathbb{Z})\cap \rho$, which is the union of coordinates axis and the infinite line. Since there is a 3-cell $\tilde{\rho}\subset \Delta_{d}^o$ such that $\rho \subset\partial \tilde{\rho}$, $X_\rho \cap X_{\tilde{\rho}}$ is given by $-Z_{m^0}=Z_{m^1}+Z_{m^2}+Z_{m^3}=0$ where $m^0$ is the vertex of $(\tilde{\rho}\backslash \rho)\cap \Delta_{d}^o(\mathbb{Z})$. Hence the regular locus  $X_\rho^o$ of $X_\rho$ is a pair-of-pants.
\end{proof}

\begin{lemma}\label{le5}
Let $\rho$ be a 3-cell of   $\mathcal{T}_v$ such that $\rho \subset \Delta_d^o$,  $\{m^0, m^1, m^2, m^3\}=\rho\cap \Delta_d^o(\mathbb{Z})$, and $\ell_\rho$ be the supporting function of $\rho$.
  \begin{enumerate}
\item[i)]For any $m\in \Delta_d(\mathbb{Z})$, $$w^{m}=t^{\ell_\rho(m)}Z_{m^0}^{a_0}Z_{m^1}^{a_1}Z_{m^2}^{a_2}Z_{m^3}^{a_3},$$ where $m=a_0m^0+ a_1m^1+ a_2m^2+ a_3m^3$, and  $ a_0+ a_1+ a_2+ a_3=1$.
    \item[ii)] If $m\in \Delta_d^o(\mathbb{Z})$,  then $$Z_{m^0}^{b_0}Z_{m^1}^{b_1}Z_{m^2}^{b_2}Z_{m^3}^{b_3}Z_m=t^{\ell_\rho(m)-v(m)}Z_{m^0}^{a_0+b_0}Z_{m^1}^{a_1+b_1}Z_{m^2}^{a_2+b_2}Z_{m^3}^{a_3+b_3},$$ where $b_i=\max\{0, -a_i\}$, $i=0,1,2,3$.
    \item[iii)] If $\rho'$  is a 3-cell
  sharing a 2-face with $\rho$, i.e.  $\rho\cap \rho' $ is a 2-cell in $\mathcal{T}_v$,
    and  $\{m^1,m^2,m^3,m^4\}=\rho' \cap \Delta_d(\mathbb{Z})$, then $$ w^{m^4+m^0}=w^{\epsilon_1m^1+ \epsilon_2m^2+ \epsilon_3m^3},   \  \  and  \  \   w^{m^4}Z_{m^0} =t^{\ell_\rho(m^4)}Z_{m^1}^{\epsilon_1}Z_{m^2}^{\epsilon_2}Z_{m^3}^{\epsilon_3},$$ where two elements of $\{\epsilon_1, \epsilon_2, \epsilon_3\}$   equal to $1$, and one element is $0$, e.g. $\epsilon_1=\epsilon_2=1$ and $\epsilon_3=0$.
     \end{enumerate}
\end{lemma}

\begin{proof}  Note that  $\rho$  is a standard simplex generated by a $\mathbb{Z}$-basis $e_1, e_2, e_3\in \mathbb{Z}^3$, i.e.  $\rho=\{x_1e_1+x_2e_2+x_3e_3\in \mathbb{R}^3|x_i\geq 0, i=1,2,3, x_1+x_2+x_3\leq 1\}$ and $\mathbb{Z}^3=\mathbb{Z}\cdot e_1+\mathbb{Z}\cdot e_2+\mathbb{Z}\cdot e_3$. For example, $e_1=(1,-1,0)$, $e_2=(1,0,0)$ and $e_3=(0,-1,1)$ generate  a simplex in $\mathcal{T}_v$.

If we regard   $m^0$ as the origin,   then $e_i=m^i-m^0$, $i=1,2,3$, is the $\mathbb{Z}$-basis. Therefore, for any $m\in \Delta_d(\mathbb{Z})$, $m=m^0+a_1e_1+a_2e_2+a_3e_3$ where $a_i \in \mathbb{Z}$.  Since $w^{m^j}=w^{m^0}w^{e_j}$, $j=0,1,2,3$, we have  \begin{align*} w^m &=w^{m^0}\prod_{i=1}^{3}(w^{e_i})^{a_i}\\ &= t^{v(m^0)}Z_{m^0}\prod_{i=1}^{3}(t^{v(m^i)-v(m^0)}Z_{m^i}/Z_{m^0})^{a_i}. \end{align*} Let  $\ell_\rho (x)=\langle n_\rho, x \rangle+b_\rho$ be  the supporting function of the cell $\rho$, i.e. $\ell_\rho(m^j)=v(m^j)$, $j=0,1,2,3$, and $v(m)> \ell_\rho(m)$ for any $m\in  \Delta_d(\mathbb{Z})\backslash \rho$. Then $$ v(m^i)-v(m^0)=\langle n_\rho, e_i\rangle, \ \ {\rm  and } \  \  \ell_\rho(m)=v(m^0)+\langle n_\rho, \sum\limits_{i=1}^{3}a_ie_i\rangle.$$  Thus $$ w^m =t^{\ell_\rho(m)}Z_{m^0}^{a_0}Z_{m^1}^{a_1}Z_{m^2}^{a_2}Z_{m^3}^{a_3},  $$  where $a_0=1-\sum\limits_{i=1}^{3}a_i$.

If we write $m^4=\epsilon_0m^0+ \epsilon_1m^1+ \epsilon_2m^2+ \epsilon_3m^3$, then $m^4-m^0=\sum\limits_{i=1}^{3}\epsilon_i(m^i-m^0)$. Note that $\rho'$ belongs to the cube generated by $m^i-m^0$,  i.e. $0\leq \epsilon_i\leq 1$,  for $i=1,2,3$,  and $\rho \cap \rho'$ belongs to the hyperplane  $\{\sum\limits_{i=1}^{3}x_i(m^i-m^0)|x_1+x_2+x_3=1\}$. Thus
 $ \epsilon_0=1- \epsilon_1- \epsilon_2-\epsilon_3<0$,  and $\sum\limits_{i=1}^{3}(m^i-m^0)$ belongs to a 3-cell sharing no faces with $\rho$.  We assume that $\epsilon_0=-1$,  $\epsilon_1=\epsilon_2=1$, and $\epsilon_3=0$ without  loss of generality. Therefore  $$w^{m^4+m^0}=w^{\epsilon_1m^1+ \epsilon_2m^2+ \epsilon_3m^3},$$ and
 we obtain the conclusion.
\end{proof}

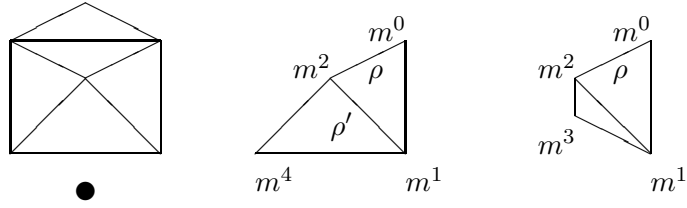
\begin{figure}\label{pic+}
{\footnotesize
\caption{ $m^0\in \rho$ vs $m^4\in\rho'$, and $w^{m^4}Z_{m^0} =t^{\ell_\rho(m^4)}Z_{m^2}Z_{m^3}$.}}
 \begin{center}
  \setlength{\unitlength}{0.5cm}
  \begin{picture}(6,6)(0,0)
\put(0,1){\line(0,1){3}}
\put(0,1){\line(1,0){4}}
\put(0,4){\line(2,-1){2}}
\put(0,1){\line(1,1){2}}
\put(2,3){\line(2,1){2}}
\put(0,4){\line(2,1){2}}
\put(4,1){\line(0,1){3}}
\put(2,5){\line(2,-1){2}}
\put(4,1){\line(-1,1){2}}
\put(0,4){\line(1,0){4}}
\put(2,0){\circle*{0.5}}
\end{picture}
  \setlength{\unitlength}{0.5cm}
  \begin{picture}(6,5)(0,0)
\put(0,0){$m^4$}
\put(0,1){\line(1,0){4}}
\put(1,3){$m^2$}
\put(0,1){\line(1,1){2}}
\put(2,3){\line(2,1){2}}
\put(4,0){$m^1$}
\put(4,1){\line(0,1){3}}
\put(3,4){$m^0$}
\put(4,1){\line(-1,1){2}}
\put(3,3){$\rho$}
\put(2,1.5){$\rho'$}
\end{picture}
 \setlength{\unitlength}{0.5cm}
  \begin{picture}(6,5)(0,0)
  \put(1,3){$m^2$}
  \put(1,1){$m^3$}
\put(2,3){\line(2,1){2}}
\put(4,1){\line(0,1){3}}
\put(4,1){\line(-1,1){2}}
\put(2,3){\line(0,-1){1}}
\put(3,3){$\rho$}
\put(3,4){$m^0$}
\put(4,0){$m^1$}
\put(2,2){\line(2,-1){2}}
\end{picture}
 \end{center}
 \end{figure}

Note that if $d=5$, then the canonical bundle $\mathcal{O}_{X_t}(K_{X_t})=\mathcal{O}_{\mathbb{CP}^3}(1)|_{X_t}$, and $X_t \subset \mathbb{CP}^3$ is the canonical embedding.  Thus vi)  in  Theorem   \ref{mainthm} holds as shown  in the introduction.
Now we assume $d>5$. Then $\Delta_d^o$ contains  $(d-4)^3$ 3-cells of $\mathcal{T}_v$.

\begin{lemma}\label{le6}
Let $\rho\subset \Delta_d^o$ be a 3-cell in $\mathcal{T}_v$, and $\rho'$  be another 3-cell
  sharing a 2-face with $\rho$, i.e.  $\rho\cap \rho' $ is a 2-cell in $\mathcal{T}_v$.  Denote $\{m^0, m^1,m^2,m^3\}=\rho \cap \Delta_d(\mathbb{Z})$ and $\{m^1,m^2,m^3,m^4\}=\rho' \cap \Delta_d(\mathbb{Z})$.
   \begin{enumerate}
\item[i)] If $\rho' \subset \Delta_d^o$, then there are integers  $c_1\geq 0, c_2\geq 0, c_3\geq 0$ such that $$Z_{m^1}^{c_1}Z_{m^2}^{c_2}Z_{m^3}^{c_3}\Big(\sum_{m \in\Delta_d^o(\mathbb{Z})}Z_m\Big)+o_1(t)=0,$$ where $o_1(t)$ is a polynomial.
\item[ii)]  Assume that $m^4 \in \partial \Delta_d$.   If  $\rho''\subset\Delta_d^o$ is another 3-cell such that $m^0$ is not a vertex of $\rho''$, i.e.  $\rho'' \cap \Delta_d(\mathbb{Z})= \{m^5, m^6,m^7, m^8\}$ and $m^0\neq m^j$, $j=5,6,7,8$,
then there are  integers $c_5\geq 0$, $c_6\geq 0$, $c_7\geq 0$, $c_8\geq 0$ such that $$Z_{m^5}^{c_5}Z_{m^6}^{c_6}Z_{m^7}^{c_7}Z_{m^8}^{c_8}Z_{m^0}\Big(\sum_{m \in\Delta_d^o(\mathbb{Z})}Z_m\Big)+o_2(t) =0,$$ for a polynomial $o_2(t)$.
\item[iii)]  Both $o_1(t)$ and $o_2(t)$ are polynomials of $Z_m$ with the coefficients tending  to zero when  $t\rightarrow\infty$, i.e.
 $$o_1(t)\rightarrow 0,  \  \   and \  \ o_2(t)\rightarrow 0.$$
\end{enumerate}
\end{lemma}

\begin{proof}
The patchworking polynomial reads
 $$ f_t=\sum_{m \in\Delta_d(\mathbb{Z})} t^{-v(m)}w^m=\sum_{m \in\Delta_d^o(\mathbb{Z})}Z_m+ \sum_{
 m' \in \partial\Delta_d\cap \mathbb{Z}^3 } t^{-v(m')}w^{m'}=0,$$ by $\Delta_d^o(\mathbb{Z})= {\rm int}(\Delta_d)\cap \mathbb{Z}^3$.

   Assume that  there is another 3-cell $\rho' \subset \Delta_d^o$ such that $\rho\cap \rho'$ is the 2-cell with vertices  $m^1, m^2, m^3$.
    If $m^4$ denotes the other vertex of $\rho'$, then
  by Lemma \ref{le5}, $w^{m^4+m^0}=w^{\epsilon_1m^1+ \epsilon_2m^2+ \epsilon_3m^3}$, or equivalently,  $$Z_{m^4}Z_{m^0} =t^{\ell_\rho(m^4)-v(m^4)}Z_{m^1}^{\epsilon_1}Z_{m^2}^{\epsilon_2}Z_{m^3}^{\epsilon_3},$$ for some $0\leq \epsilon_i\leq 1$, $i=1,2,3$,  and   $$ t^{-v(m')}w^{m'}= t^{\ell_\rho(m')-v(m')}Z_{m^0}^{a_0}Z_{m^1}^{a_1}Z_{m^2}^{a_2}Z_{m^3}^{a_3}= t^{\ell_{\rho'}(m')-v(m')}Z_{m^4}^{b_0}Z_{m^1}^{b_1}Z_{m^2}^{b_2}Z_{m^3}^{b_3},$$ for certain $a_i$ and $b_i$, $i=0,1,2,3$, where $m'=a_0m^0+\sum\limits_{i=1}^{3}a_im^i= b_0m^4+\sum\limits_{i=1}^{3}b_im^i$, and  $a_0+b_0= 0$.
   We choose the monomial with the power of $ Z_{m^0}$ or $ Z_{m^4}$ being non-negative in the above expression, i.e. if $a_0\geq 0$ without loss of generality,  $$ t^{-v(m')}w^{m'}= t^{\ell_\rho(m')-v(m')}Z_{m^0}^{a_0}Z_{m^1}^{a_1}Z_{m^2}^{a_2}Z_{m^3}^{a_3}.$$ Equivalently, we choose either the cell $\rho$ or $\rho'$ such that the only possible poles of  $w^{m'}$ are  along $\{Z_{m^i}=0\}$, $i=1,2,3$.  Thus there are $c_1\geq 0, c_2\geq 0, c_3\geq 0$ such that $$Z_{m^1}^{c_1}Z_{m^2}^{c_2}Z_{m^3}^{c_3}\Big(\sum_{m \in\Delta_d^o(\mathbb{Z})}Z_m\Big)+\sum_{
 m' \in \partial\Delta_d\cap \mathbb{Z}^3 } t^{\ell(m')-v(m')}M_{m'}=0,$$  where $M_{m'}$ is a monomial, and either $\ell=\ell_{\rho}$ or $\ell=\ell_{\rho'}$.  Let $$o_1(t)=\sum_{
 m' \in \partial\Delta_d\cap \mathbb{Z}^3 } t^{\ell(m')-v(m')}M_{m'},$$ which satisfies $o_1(t)\rightarrow 0$ as polynomials  when $t\rightarrow\infty$,
  since $t^{\ell_\rho(m')-v(m')}$ and $t^{\ell_{\rho'}(m')-v(m')}$ have  negative powers.

Now we assume that $m^4\in \partial\Delta_d$.  By  Lemma \ref{le5}, $$Z_{m^0}\Big(\sum_{m \in\Delta_d^o(\mathbb{Z})}Z_m+\sum_{m^4\neq  m' \in \partial\Delta_d\cap \mathbb{Z}^3 } t^{-v(m')}w^{m'}\Big)+ t^{\ell_\rho(m^4)-v(m^4)}Z_{m^1}^{\epsilon_1}Z_{m^2}^{\epsilon_2}Z_{m^3}^{\epsilon_3} =0.$$ If  $\rho''\subset\Delta_d^o$ is another 3-cell such that $m^0$ is not a vertex of $\rho''$, i.e.  $\rho'' \cap \Delta_d(\mathbb{Z})= \{m^5, m^6,m^7, m^8\}$ and $m^0\neq m^j$, $j=5,6,7,8$,  then by Lemma \ref{le5}, $$t^{-v(m')}w^{m'}=t^{\ell_{\rho''}(m')-v(m')} Z_{m^5}^{a_5'}Z_{m^6}^{a_6'}Z_{m^7}^{a_7'}Z_{m^8}^{a_8'}. $$ The same argument as above shows that there are integers $c_5\geq 0$, $c_6\geq 0$, $c_7\geq 0$, $c_8\geq 0$ such that $$Z_{m^5}^{c_5}Z_{m^6}^{c_6}Z_{m^7}^{c_7}Z_{m^8}^{c_8}Z_{m^0}\Big(\sum_{m \in\Delta_d^o(\mathbb{Z})}Z_m\Big)+o_2(t) =0,$$ where  $$o_2(t)=\sum_{
 m' \in \partial\Delta_d\cap \mathbb{Z}^3 } t^{\ell''(m')-v(m')}M_{m'}'\rightarrow 0,$$  when $t\rightarrow\infty$,  $M_{m'}'$ are monomials of $Z_{m^0}, Z_{m^1}, Z_{m^2}, Z_{m^3}, Z_{m^5}, Z_{m^6}, Z_{m^7},Z_{m^8}$, and either $\ell''=\ell_{\rho}$  or $\ell''=\ell_{\rho''}$.
\end{proof}

\begin{proof}[Proof  of vi)  in  Theorem   \ref{mainthm}]
When $t\rightarrow \infty$,   $\Psi_t(X_t)$ converges to a analytic subset  $X_\infty$  of  dimension 2   in $\mathbb{CP}^N$ by passing to a subsequence  if necessary  (cf. \cite{Bi}). The convergence is in the sense of analytic spaces.  Recall that an analytic subset   $Y$ is locally defined by  finite holomorphic functions
$u_1=0, \cdots, u_k=0$, and the function sheaf $\mathcal{O}_Y $ on $Y$ is locally given by the quotient  $\mathcal{O}_{\mathbb{CP}^N}/(u_1, \cdots, u_k)$. Note that nilpotent functions are allowed. An analytic subset   $Y'\subset Y$ is a subset of $Y$ with surjections $\mathcal{O}_Y \rightarrow \mathcal{O}_{Y'} $. A family of subsets $Y_t$ converges to $Y$, if
 $Y_t$ is locally given by $u_1^t=0, \cdots, u_k^t=0$, and $u_j^t\rightarrow u_j$, $j=1, \cdots, k$,  when $t\rightarrow \infty$. For example, if $Y$ is defined by $x^2=0$ in $\mathbb{C}^2$, and $Y'$ is given by $x=0$, then $Y'\subset Y$ via $\mathbb{C}[x,y]/(x^2)\rightarrow \mathbb{C}[x,y]/(x)$. We could  view  $Y$ as two copies of $Y'$ stacking together, i.e. $Y=2Y'$. A family $Y_t$, defined  by $x(x+t^{-1}y)=0$, converges to $Y$.

 For any $m\in \Delta_d^o(\mathbb{Z})$, let $H_m$ be the hyperplane in $\mathbb{CP}^N$ defined by $Z_m=0$, and for any 3-cell $\rho \subset \Delta_d^o$ of $\mathcal{T}_v$, let $\tilde{X}_\rho$ be the subset given by $Z_{m'}=0$ for all $m'\in \Delta_d^o(\mathbb{Z})\backslash \rho$, which is the  projective space of dimension $3$, i.e. $$
\tilde{X}_\rho=\{[Z_{m^0},Z_{m^1},Z_{m^2},Z_{m^3}]\in \mathbb{CP}^3\}=\bigcap_{m' \in \Delta_d^o(\mathbb{Z})\backslash \rho } H_{m'},$$ where $\{m^0, m^1, m^2, m^3\}=\rho\cap  \Delta_d^o(\mathbb{Z})$. Furthermore, if $\rho' \subset \Delta_d^o$ is an another 3-cell, then $$\tilde{X}_\rho \cap \tilde{X}_{\rho'}= \tilde{X}_\rho \cap H_{m''},$$ where  $m''\in (\rho'\backslash \rho)\cap \Delta_d^o(\mathbb{Z})$.

The image $\Psi_t(X_t)$ satisfies the equations in Lemma \ref{le5} and  Lemma \ref{le6}. If $\rho\subset \Delta_d^o$ with vertices  $\{m^0, m^1, m^2, m^3\}$, then ii) of Lemma \ref{le5} asserts that $X_\infty$ satisfies $Z_{m^0}^{b_0}Z_{m^1}^{b_1}Z_{m^2}^{b_2}Z_{m^3}^{b_3}Z_m=0$ for all
$m\in \Delta_d^o(\mathbb{Z})\backslash \rho$, and thus, as analytic subset, $$X_\infty\subset \tilde{X}_\rho \cup b_0'H_{m^0}\cup  b_1'H_{m^1} \cup  b_2'H_{m^2}\cup b_3'H_{m^3},$$   for certain $b_i'\geq 0$, $i=0,1,2,3$.  Since $\tilde{X}_\rho \subset H_m$ for all $m\in \Delta_d^o(\mathbb{Z})\backslash \rho$, we obtain $$ X_\infty\subset \bigcup_{{\rm all \ 3-cells} \ \rho \subset \Delta_d^o}\tilde{X}_\rho=\tilde{X},$$  as analytic subsets. Note that $\tilde{X}$ is a subvariety, and consists of $(d-4)^3$ irreducible components. Each irreducible component of $\tilde{X}$ is a copy of $\mathbb{CP}^3$.

  We consider two 3-cells  $\rho$ and $\rho'$ in  $ \Delta_d^o$, which
  share a 2-face, i.e.  $\rho\cap \rho' $ is a 2-cell in $\mathcal{T}_v$.  Let $\{m^0, m^1,m^2,m^3\}=\rho \cap \Delta_d(\mathbb{Z})$ and $\{m^1,m^2,m^3,m^4\}=\rho' \cap \Delta_d(\mathbb{Z})$. By  Lemma \ref{le6}, $X_\infty$ satisfies $$Z_{m^1}^{c_1}Z_{m^2}^{c_2}Z_{m^3}^{c_3}\Big(\sum_{m \in\Delta_d^o(\mathbb{Z})}Z_m\Big)=0,$$ for certain $c_i\geq 0$, $i=1,2,3$, and thus, $X_\infty\cap \tilde{X}_\rho$ is given by  $$Z_{m^1}^{c_1}Z_{m^2}^{c_2}Z_{m^3}^{c_3}\Big(Z_{m^0}+Z_{m^1}+Z_{m^2}+Z_{m^3}\Big)=0.$$ Consequently, $$X_\infty \cap \tilde{X}_\rho \subset X_\rho \cup
  c_1 (H_{m^1}\cap \tilde{X}_\rho)\cup c_2 (H_{m^2}\cap\tilde{X}_\rho)\cup c_3( H_{m^3}\cap\tilde{X}_\rho),$$ and $X_\infty \cap \tilde{X}_\rho \cap H_{m^0}\subset  X_\rho \cap H_{m^0}$, i.e. a line in the projective plane $\tilde{X}_\rho\cap  H_{m^0}$. Therefore, the plane $\tilde{X}_\rho\cap  H_{m^0}=\tilde{X}_{\rho'}\cap  H_{m^4}$ corresponding to the 2-cell $\rho\cap\rho'$ is not a irreducible component of $X_\infty$, and only intersects with $X_\infty$ at most a line.

  Assume that $\rho_2=\rho \cap \partial \Delta_d^o$ is a 2-cell in $\mathcal{T}_v$. If the  vertex $m^0\subset \rho$ does not belong to $\rho_2$, then    $X_{\rho_2}$ is given by $Z_{m^0}=0$, i.e. $ X_{\rho_2}=\tilde{X}_\rho \cap H_{m^0}$.   If  $\rho''\subset\Delta_d^o$ is another 3-cell such that $m^0$ is not a vertex of $\rho''$ with $\rho'' \cap \Delta_d(\mathbb{Z})= \{m^5, m^6,m^7, m^8\}$,
then $X_\infty\cap \tilde{X}_\rho$ satisfies  that $$Z_{m^5}^{c_5}Z_{m^6}^{c_6}Z_{m^7}^{c_7}Z_{m^8}^{c_8}Z_{m^0}\Big(Z_{m^0}+Z_{m^1}+Z_{m^2}+Z_{m^3}\Big) =0,$$ for certain  $c_j\geq 0$, $j=5,6,7,8$ by Lemma  \ref{le6}.  Hence $$X_\infty \cap \tilde{X}_\rho \subset X_{\rho_2} \cup X_\rho \cup
  c_1' (H_{m^1}\cap\tilde{X}_\rho)\cup c_2' (H_{m^2}\cap\tilde{X}_\rho)\cup c_3'( H_{m^3}\cap\tilde{X}_\rho),$$ for certain $c_j'\geq 0$. Thus $X_\infty$  is a closed   subvariety of $X_0$,  i.e. $$X_\infty \subset X_0,$$ where $X_0$ is defined  by (\ref{eq5.4}).

Since $$c_1^2(\mathcal{O}_{\mathbb{CP}^N}(1)|_{X_\infty})=c_1^2(\mathcal{O}_{\mathbb{CP}^N}(1)|_{X_t})=K_{X_t}^2=d(d-4)^2= c_1^2(\mathcal{O}_{\mathbb{CP}^N}(1)|_{X_0}),$$ there is no more irreducible   components in $X_0$ besides $X_\infty$. We obtain the conclusion $X_\infty =X_0$. \end{proof}

\end{document}